\documentclass[a4paper,reqno]{amsart}
\usepackage{hyperref}
\usepackage{mathtools}
\usepackage{tikz}
\usepackage{color}
\usepackage{enumitem}

 \allowdisplaybreaks \numberwithin{equation}{section}
\begingroup
\theoremstyle{plain}
\newtheorem{theorem}{Theorem}[section]
\newtheorem{proposition}[theorem]{Proposition}
\newtheorem{lemma}[theorem]{Lemma}
\newtheorem{corollary}[theorem]{Corollary}

\theoremstyle{definition}
\newtheorem{definition}[theorem]{Definition}
\newtheorem{remark}[theorem]{Remark}

\endgroup

\newcommand{\ba}{\mathbf{a}}
\newcommand{\bc}{\mathbf{c}}

\newcommand{\bx}{\mathbf{x}}

\def \div {\mathop {\rm div}\nolimits}
\def \curl {\mathop {\rm curl}\nolimits}
\def \dist {\mathop {\rm dist}\nolimits}
\def \spt {\mathop {\rm spt}\nolimits}
\def \e {\epsilon}
\def \a {{a}}
\def \b {{b}}

\def \oea {\Omega_\eps(\a)}
\def \re {\mathbb R}

\def \E {\mathcal E_\eps(\a)}

\def \R {\mathbb R}
\def \Om {\Omega}
\def \eps {\e }
\newcommand{\de}{\mathrm{d}}
\def \ezero {\e_0}
\def \ea{\e_\alpha}

\def \barv {\bar v^\e}
\def \spazio {\mathcal X}

\newcommand{\cF}{\mathcal{F}}

\newcommand{\mygraphic}[1]{\includegraphics[height=#1]{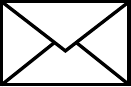}}
\newcommand{\myenv}{(\raisebox{0pt}{\mygraphic{.6em}})}

\renewcommand{\k}{K}

\title[Confinement of dislocations in a crystal]{Confinement of dislocations inside a crystal with a prescribed external strain}
\author{Ilaria Lucardesi}
\address[I.\@ Lucardesi]{Institut \'Elie Cartan de Lorraine, B.P.\@ 70239, 54506 Vandoeuvre-l\`es-Nancy, France}
\email{ilaria.lucardesi@univ-lorraine.fr}
\author{Marco Morandotti}
\address[M.\@ Morandotti]{Dipartimento di Scienze Matematiche, Politecnico di Torino, Corso Duca degli Abruzzi, 24, 10129 Torino, Italy}
\email{marco.morandotti@polito.it}
\author{Riccardo Scala}
\address[R.\@ Scala]{Dipartimento di Matematica ``'G.~Castelnuovo'', Sapienza Universit\`a di Roma, Piazzale Aldo Moro, 5, 00185 Roma, Italy}
\email{scala@mat.uniroma1.it}
\author{Davide Zucco}
\address[D.\@ Zucco \myenv]{Dipartimento di Matematica ``G.~Peano'', Universit\`a di Torino, Via Carlo Alberto, 10, 10123 Torino, Italy}
\email{davide.zucco@unito.it}

\date{\today.} 

\begin{document}

\begin{abstract}
A system of $n$ screw dislocations in an isotropic crystal undergoing antiplane shear is studied in the framework of linear elasticity.
Imposing a suitable boundary condition for the strain, namely requesting the non-vanishing of its boundary integral, results in a confinement effect.
More precisely, in the presence of an external strain with circulation equal to $n$ times the lattice spacing, it is energetically convenient to have $n$ distinct dislocations lying inside the crystal.
The result is obtained by formulating the problem via the core radius approach and by studying the asymptotics as the core size vanishes.
An iterative scheme is devised to prove the main result.
This work sets the basis for studying the upscaling problem, i.e., the limit as $n\to\infty$, which is treated in \cite{LMSZupscaling}.
\end{abstract}

\maketitle

\smallskip




\section{Introduction}
Starting with the pioneering work of Volterra \cite{volterra}, much attention has been drawn on dislocations in solids, as the ultimate cause of plasticity in crystalline materials \cite{GLP,orowan,polanyi,taylor}.
Dislocations are line defects in the lattice structure. 
The interest in dislocations became more and more evident as soon as it was understood that their presence can significantly influence the chemical and physical properties of the material.
The measure of the lattice mismatch due to a dislocation is encoded in the \emph{Burgers vector}, whose magnitude is of the order of one lattice spacing (see \cite{HB}).
According to whether the Burgers vector is perpendicular or parallel to the dislocation line, ideal dislocations are classified as \emph{edge dislocations} or \emph{screw dislocations}, respectively.
In nature, real dislocations come as a combination of these two types.
For general treaties on dislocations, we refer the reader to \cite{HL,HB,nabarro}.

In this paper we focus our attention on screw dislocations in a single isotropic crystal which occupies a cylindrical region $\Om\times\R$ and which undergoes antiplane shear.
According to the model proposed in \cite{CG} in the context of linearized elasticity, this allows us to study the problem in the cross section $\Omega\subset\R^2$. 
Throughout the work, we will assume that
\begin{equation}\label{H1}\tag{H1}
\text{$\Omega$ is a bounded convex open set with $C^1$ boundary.}
\end{equation}
We consider the lattice spacing of the material to be $2 \pi$ and that all the Burgers vectors are oriented in the same direction. 
Therefore, every dislocation line is directed along the axis of the cylinder, is characterized by a Burgers vector of magnitude $2\pi$ along the same axis, and meets the cross section $\Om$ at a single point. 
Moreover, we assume that an \emph{external strain} acts on the crystal: we prescribe the tangential strain on $\partial \Om$ to be of the form
\begin{equation}\label{datum}\tag{H2}
f\in L^{1}(\partial \Om)\quad \text{with}\quad \int_{\partial \Om} f(x)\,\de \mathcal H^1(x)=2\pi n
\end{equation}
for some $n \in \mathbb N$. This choice of the external strain will determine at most $n$ \emph{distinct} dislocations inside $\Om$ (see, e.g., \cite{BBH,SaSe} for a comment on the topological necessity of the presence of exactly $n$ defects; {see also \cite{DPV2015}, where an evolution problem in the fractional laplacian setting is also studies}), which we denote by $\ba\coloneqq(a_1,\ldots,a_n)\in\Omega^n\setminus\triangle_n$, where $\triangle_n\coloneqq\{\bx=(x_1,\ldots,x_n)\in\Omega^n: \text{there exist $i\neq j$ such that $x_i=x_j$}\}$.
 
Since the elastic energy associated with a defective material is infinite and has a logarithmic explosion in the vicinity of each dislocation $a_i$ (see, e.g., \cite{CG,nabarro}),
we resort to the so-called \emph{core radius approach}, which consists in considering the energy far from the dislocations $a_1,\ldots,a_n$. 
More precisely, given $\eps>0$, we aim at studying the elastic energy in the perforated domain $\Omega_\eps(\ba)\coloneqq \Omega\setminus\bigcup_{i=1}^n\overline{B}_\eps(a_i)$.
This approach is standard in the literature and it is employed in different contexts such as linear elasticity (see, for instance, \cite{nabarro,TOP,VKLLO}; also \cite{BM,CG,pons} for screw dislocations and \cite{CL} for edge dislocations), the theory of Ginzburg-Landau vortices (see, for instance, \cite{BBH,SaSe} and the refereces therein), and liquid crystals (see, for instance, \cite{GSV}).
Since the core radius approach eliminates the non-integrability of the strain field around the dislocations, classical variational techniques can be used.
Therefore, we consider the energy 
\begin{equation}\label{energy-n}
\mathcal E_\e^{(n)}(\ba):=
\min \bigg\{\frac{1}{2}\int_{\Omega_\e(\ba)} |F|^2\,\de x\, : \text{$F\in \spazio_\e(\ba) $, $F\cdot \tau=f$ on $\partial \Om\setminus \bigcup_{i=1}^n \overline{B}_\e(a_i)$}\bigg\}\,,
\end{equation}
where $\tau$ denotes the tangent unit vector to $\partial \Omega$ and
\begin{equation*}
\spazio_\e(\ba) \coloneqq \Big\{F\in L^2(\Omega_\e(\ba); \re^2): \curl F = 0\;  \hbox{in }\mathcal D'(\Omega_\e(\ba)), \, \langle F\cdot \tau,1 \rangle_\gamma = 2\pi m \Big\},
\end{equation*}
where $\mathcal{D}'(\Omega_\e(\ba))$ is the space of distributions on $\Omega_\e(\ba)$ and $\gamma$ is an arbitrary simple closed curve in $ \Omega_\e(\ba)$ winding once counterclockwise around $m$ dislocations. 
Note that the boundary condition $F\cdot \tau$ must be intended in the sense of traces and that $\langle \cdot ,\cdot \rangle_\gamma$ denotes the duality between $H^{-1/2}(\gamma)$ and $H^{1/2}(\gamma)$ (see \cite{ChenFrid}).
Since these spaces are encapsulated, the energy $\mathcal{E}_\e$ is monotone: 
if $0<\e<\eta$ then 
\begin{equation}\label{ginevra}
\mathcal E_\e^{(n)}(\ba)\geq\mathcal E_\eta^{(n)}(\ba).
\end{equation}

If the $a_i$'s are all distinct and inside $\Om$, the energy \eqref{energy-n} scales like $\pi n|\log\e|$.
This suggests to study the asymptotic behavior, as $\e\to0$, of the functionals $\mathcal F_\e^{(n)}\colon\overline\Om{}^n\to\R\cup\{+\infty\}$ defined by
\begin{equation}\label{Fepsn}
\mathcal F_\e^{(n)}(\ba):= \mathcal E_\e^{(n)}(\ba)-\pi n|\log\eps|.
\end{equation}
In this context, we say that the sequence of functionals $\mathcal F_\e^{(n)}$ \emph{continuously converges} in $\overline\Om{}^n$ to $\mathcal F^{(n)}$ as $\e\to0$ if, for any sequence of points $\ba^\e\in\overline\Om{}^n$ converging to $\ba\in\overline\Om{}^n$, the sequence (of real numbers) $\mathcal F_\e^{(n)}(\ba^\e)$ converges to $\mathcal F^{(n)}(\ba)$.

In order to write the limit functional, we introduce two objects: we take $g\colon \partial\Om\to\R$ a primitive of $f$ with $n$ jump points $b_i\in \partial \Omega$ and jump amplitude $2\pi$ (see \eqref{sum-gi} for a precise definition), and
 for every $i\in \{1,\ldots,n\}$ we set 
\begin{equation}\label{dconi}
d_i:=\min_{j\in \{1,\ldots,n \}\atop j\neq i}\bigg\{\frac{|\a_i-\a_j|}2,\dist(a_i,\partial \Om)\bigg\}\,,
\end{equation}
where $\mathrm{dist}(\cdot, \partial \Omega)$ is the distance function from $\partial \Omega$.
\begin{theorem}\label{thGcn}
Under the assumptions \eqref{H1} and \eqref{datum}, as $\e\to0$ the functionals $\mathcal F_\e^{(n)}$ defined by \eqref{Fepsn} continuously converge in $\overline\Om{}^n$ to the functional $\mathcal F^{(n)}\colon\overline\Om{}^n\to\R\cup\{+\infty\}$ defined as
\begin{equation}\label{Gammalimitn}
\begin{split}
\mathcal F^{(n)}(\ba)\coloneqq \sum_{i=1}^n\pi\log d_i&+\frac{1}{2}\int_\Om |\nabla v_{\ba}|^2\,\de x+\sum_{i=1}^n\frac12\int_{\Om_{d_i}(a_i)} |\k_{a_i}|^2\,\de x \\
&+\sum_{i=1}^n\int_{\Om_{d_i}(a_i)} \nabla v_{\ba}\cdot \k_{a_i}\,\de x+\sum_{i<j} \int_\Om \k_{a_i}\cdot \k_{a_j}\,\de x,
\end{split}
\end{equation}
if $\ba\in\Om^n\setminus\triangle_n$, 
and $\mathcal F^{(n)}(\ba)\coloneqq +\infty$ otherwise. 
Here $\k_{a_i}(x)\coloneqq \rho_{a_i}^{-1}(x)\hat\theta_{a_i}(x)$, being $(\rho_{a_i}, \theta_{a_i})$ a system of polar coordinates centered at $a_i$, and $v_{\ba}$ solves 
\begin{equation*}
\begin{cases}
\Delta  v_{\ba}=0&\text{ in  }\Omega, \\
v_{\ba}=g-\sum_{i=1}^n\theta_{\a_i} &\text{ on  }\partial\Omega.
\end{cases}  
\end{equation*}
In particular, $\mathcal F^{(n)}$ is continuous in $\overline\Om{}^n$ and diverges to $+\infty$ if either at least one dislocation approaches the boundary or at least two dislocations collide, that is, $\mathcal F^{(n)}(\ba)\to +\infty$ as $d_i\to0$ for some $i$. 
Thus, $\mathcal F^{(n)}$attains its minimum in ${\Om}{}^n\setminus\triangle_n$.
\end{theorem}

We notice that, by rotating the vector fields of $\pi/2$,  \eqref{Gammalimitn} can also be expressed as the sum of two terms: the \emph{self energy}  $\mathcal E_\text{self}$, responsible for the contribution of individual dislocations, and the \emph{interaction energy} $\mathcal E_\text{int}$, depending on 
the mutual position of two dislocations. In formulas, 
\[
\mathcal E_\text{self}(a_i)\coloneqq\pi\log \mathrm{dist}(a_i,\partial \Omega)+\frac{1}{2}\int_{\Om_{d(\a)}(\a)}|\nabla \phi_i + \nabla w_i|^2\, \de x+\frac{1}{2}\int_{B_{d(\a)}(\a)}|\nabla w_i|^2\,\de x,
\]
and
\[
\mathcal E_\text{int}(a_i,a_j) \coloneqq\int_\Omega (\nabla \phi_i+\nabla w_i)\cdot(\nabla \phi_j+\nabla w_j)\,\de x.
\]
Here, $\phi_i$ and $w_i$ are the solutions (with $w_i$ determined up to an additive constant) to
\begin{equation*}
\begin{cases}
\Delta \phi_i=2\pi\delta_{a_i} & \text{in $\Omega$,} \\
\phi_i(x)=\log|x-a_i| & \text{on $\partial\Omega$,}
\end{cases}
\qquad\text{and}\qquad
\begin{cases}
\Delta w_i=0 & \text{in $\Omega$,} \\
\partial_\nu w_i(x)=\tfrac1n {f}-\partial_\nu\phi_i & \text{on $\partial\Omega$.}
\end{cases}
\end{equation*}

A consequence of Theorem \ref{thGcn} is that also the energies \eqref{energy-n} attain their minimum in 
$\Om{}^n$ at an $n$-tuple of well separated points.
\begin{corollary}\label{confinon}
Under the assumptions \eqref{H1} and \eqref{datum}, there exists $\ezero>0$ such that, for every $\e\in(0,\ezero)$, the infimum problem
\begin{equation}\label{LMSZ2}
\inf\{\mathcal E_\e^{(n)}(\ba): \ba \in \overline\Omega{}^n\},
\end{equation}
admits a minimizer only in $\Om{}^n\setminus\triangle_n$.
Moreover, if $\ba^\eps \in \Omega^n\setminus\triangle_n$ is a minimizer for \eqref{LMSZ2}, then (up to subsequences) we have $\ba^\eps\to\ba$ and $\mathcal F_\e^{(n)}(\ba^\e)\to \mathcal F^{(n)}(\ba)$, as $\e\to 0$, where $\ba$ 
is a minimizer of the functional $\mathcal F^{(n)}$ defined in \eqref{Gammalimitn}. 
In particular, for $\e$ small enough, all the minimizers of problem \eqref{LMSZ2} are $n$-tuples of distinct points that stay uniformly (with respect to $\e$) far away from the boundary and from one another. 
\end{corollary}


Throughout the paper, we will always assume \eqref{H1} and \eqref{datum}, even if it is not explicitly stated.
We stress that the convexity and regularity assumptions on $\Omega$ stated in \eqref{H1} provide a 
\emph{uniform interior cone condition} of angle between $\pi/2$ and $\pi$, i.e., 
\begin{equation}\label{H3}
\begin{split}
& \text{there exist $\pi/2<\alpha<\pi$ and $\ea>0$ such that for every $\b\in\partial\Om$ the disk $B_{\ea}(b)$} \\
& \text{meets $\partial\Omega$ at two points $b_1$ and $b_2$ forming an angle at least $\alpha$ with $\b$.}
\end{split}
\end{equation}
We point out that convexity and regularity play different roles: the former is conveniently assumed in order to simplify the exposition of the results (in fact, it can be removed without changing their essence); the latter, on the other hand, is fundamental in our proofs.
Finally, we observe that the boundary condition of Dirichlet type $F\cdot\tau=f$, with $f$ as in \eqref{datum}, is fundamental to keep the dislocations confined inside the material.  
In fact, the natural boundary conditions of Neumann type imply that the dislocations migrate to the boundary and leave the domain, since, in such a case, the Dirichlet energy of the system decreases as the dislocations approach $\partial\Om$ (see, e.g., \cite{BFLM,HM}).

A key feature in our analysis is the rescaling introduced in \eqref{Fepsn} (see also \cite{ADLGP}),
which is related to the so-called \textit{Hadamard finite part} of a divergent integral (see \cite{Had}). 
Such type of asymptotic analysis has the advantage of keeping into account the energetic dependence on the position of the dislocations $\ba\in\Om^n$, whereas it is well-known that the standard rescaling obtained by dividing the energy by $|\log\eps|$ gives rise to an energy which only counts the number of dislocations in the bulk (see again \cite{ADLGP}).

Some
results close to those presented in this paper 
can be found in the literature about Ginzburg-Landau vortices (see \cite{BBH} and also \cite{SaSe}). In this respect, the present paper can be considered as a self-contained presentation of the asymptotic results for the energy \eqref{Fepsn}, presented in a language that is  familiar to the dislocation community,
targeted, in particular, to applied mathematicians and continuum mechanists. The statement of Theorem~\ref{thGcn} is in fact similar to that of \cite[Theorems~I.9 and I.10]{BBH}, but its proof is based on an original iterative procedure (close, in spirit, to some combinatoric algorithms) which makes it quite different and useful for numerical applications. The need for such an algorithm is dictated by the fact that in the core-radius approach we allow for the cores around each dislocation to intersect with one another and with the boundary of the domain, which was avoided in \cite{BBH} by introducing a safety radius.
Moreover, we set here the bases and the notation for tackling the more challenging problem of the upscaling of the system of dislocations. 
In \cite{LMSZupscaling} we study the limit as $n\to\infty$ and obtain a limit energy functional defined on measures describing the distribution of dislocations in the material.

\smallskip

{Section~\ref{preliminaries} sets the notation and presents some preparatory results.}
{Section~\ref{aux} contains a result on the properties of $\k$ and on some a priori bounds for harmonic functions.}
Section~\ref{twomeglcheone} is devoted to proving Theorem~\ref{thGcn} and Corollary~\ref{confinon}. {Section~\ref{numerica} contains numerical plots of $\mathcal F^{(n)}$ in $\Omega=B(0,1)$ under different boundary conditions.

\section{Preliminaries}\label{preliminaries}
In Subsection~\ref{notation} we introduce the notation used throughout the paper. Then, in view of the core radius approach, in Subsection \ref{ss-dfn}  we rewrite the energy $\mathcal{E}_\e^{(n)}$ in terms of the displacement of a regular function.
\subsection{Notation}\label{notation}
\begin{itemize}[leftmargin=*]
\item[-] $B_r(x)$ denotes the open disk of radius $r>0$ centered at $x\in\R^2$; $\overline B_r(x)$ is its closure;
\item[-] for $n\geq1$, and for $x_1,\ldots,x_n\in\overline\Omega$, we denote $\bx\coloneqq(x_1,\ldots,x_n)$ whenever there is no risk of misunderstanding; the symbol $\Om_r(\bx)$ denotes the open set
$$\Om_r(\bx):=\Om\setminus\bigg(\bigcup_{i=1}^n \overline B_r(x_i)\bigg);$$
$\overline\Om_r(\bx)$ is its closure;
\item[-] the function $\R^2\ni x\mapsto\dist(x,E)$ denotes the distance of $x$ from a set $E\subset\R^2$; in the particular case $E=\partial\Om$, we define $d(x)\coloneqq \dist(x,\partial\Om)$;
\item[-] we denote $\Omega^{(r)}\coloneqq\{x\in\Omega : d(x)>r\}$ and by $\overline\Omega{}^{(r)}$ its closure;
\item[-] $\Omega^n$ denotes the cartesian product of $n$ copies of $\Omega$ and $\overline\Omega{}^n$ its closure;
\item[-] for $0<r<R$, $A_r^R(x)\coloneqq B_R(x)\setminus\overline B_r(x)$ denotes the open annulus of internal radius $r$ and external radius $R$ centered at $x\in\R^2$;
\item[-] $\chi_E$ denotes the characteristic function of $E$: $\chi_E(x)=1$ if $x\in E$; $\chi_E(x)=0$ if $x\notin E$;
\item[-] given a set $E$ with piecewise $C^1$ boundary, $\nu$ and $\tau$ denote the outer unit normal and the tangent unit vectors to $\partial E$, respectively; 
\item[-] $\mathrm{diam}\, E$ denotes the diameter of a set $E\subset\R^2$;
\item[-] given $x=(x_1,x_2)\in\R^2$, we denote by $x^\perp\coloneqq(-x_2,x_1)$ the rotated vector;
\item[-] given $x\in\R^2$, we define $(\rho_x,\theta_x)$ as the standard polar coordinate system centered at $x$; 
$\hat\rho_x$ and $\hat\theta_x$ denote the corresponding unit vectors;
\item[-] given $x\in\R^2$, we denote by $\k_x$ the vector field 
$\k_x:=\rho_x^{-1}\hat\theta_x$. It is easy to see that (\cite{muskhelishvili})
\begin{equation}\label{euler-k}
\begin{cases}
\div \k_x=0& \text{in $\mathcal{D}'(\R^2)$,} \\
\curl \k_x=2\pi\delta_{x}& \text{in $\mathcal{D}'(\R^2)$,} \\
\k_x\cdot\nu=0 & \text{on $\partial B_r(x)$, for any $r>0$.}
\end{cases}
\end{equation}
Notice that $\k_x$ is the absolutely continuous part of the gradient of the function $\theta_x$, and the jump set of $\theta_x$ is the half line starting from $x$ and passing through $y$ with $[\theta_x]=2\pi$ across it;
\item[-] we define the function $\omega\colon\overline\Om \to \{1,2\}$ by
\begin{equation}\label{omeghina}
\omega(a)=\begin{cases}
1 & \text{if $a\in\Omega$,} \\
2 & \text{if $a\in\partial\Omega$;}
\end{cases}
\end{equation} 
\item[-] $\spt \varphi$ denotes the support of the function $\varphi$; 
\item[-] given $x\in\R^2$, $\delta_x$ is the Dirac measure centered at $x$;
\item[-] $\mathcal H^1$ denotes the one-dimensional Hausdorff measure;
\item[-] the letter $C$ alone represents a generic constant (possibly depending on $\Omega$) whose value might change from line to line. 
\end{itemize}

\subsection{Displacement formulation}\label{ss-dfn}
We start by introducing the multiplicity of a dislocation.

\begin{definition}[multiplicity of a dislocation]\label{def-mi}
Let $n\in \mathbb N$ with $n\geq 2$ and let $\a_1,\dots,\a_n\in  \overline\Om$. 
We label the dislocations in such a way that the first $\ell$ of them ($\ell\leq n$) are all distinct, so that $a_i\neq a_k$ for all $i,k\in\{1,\ldots,\ell\}$, $i\neq k$, and, for every $j\in\{\ell+1,\ldots,n\}$ there exists $i(j)\in\{1,\ldots,\ell\}$ such that $a_j=a_{i(j)}$.
For every $i=1,\dots,\ell$ we say that a point $a_i$ has \emph{multiplicity} $m_i$ if there are $m_i-1$ points $a_j$, with $j\in\{\ell+1,\ldots,n\}$ that coincide with $a_i$.
Clearly, $\sum_{i=1}^\ell m_i=n$.
\end{definition}
Given $n$ dislocations $a_1,\ldots,a_n\in\overline \Om$, we choose $n$ (closed) segments $\Sigma_1,\ldots,\Sigma_n$ joining the dislocations with the boundary, such that  
$\Om\setminus(\Sigma_1\cup\cdots\cup\Sigma_n)$ is simply connected. 
One possible construction of such a family of segments is to fix a point $a^*\notin \overline\Omega$ and take $\Sigma_i$ as the portion of the segment joining $a_i$ with $a^*$ lying inside $\overline\Omega$, for $i=1,\ldots,\ell$.
With this construction, if $a_i=a_j$, then $\Sigma_i=\Sigma_j$.
Moreover, we set $b_i:=\Sigma_i\cap\partial\Om$, for $i=1,\ldots,\ell$. Given 
$b\in\partial\Om\setminus\{b_1,\ldots,b_\ell\}$, we  denote by $\gamma_b^x$ the counterclockwise path in $\partial \Omega$ connecting $b$ and $x$. Such parametrization induces an ordering on the points of the boundary: we say that $x$ precedes $y$ on $\partial \Omega$, and we write $x\preceq y$, if the support of $\gamma_b^{y}$ contains that ot $\gamma_b^x$.
Without loss of generality, we may relabel the $b_i$'s (and the $a_i$'s, accordingly) so that $b\prec b_1\preceq \ldots \preceq b_\ell$.
According to this notation, for $i=1,\ldots,\ell$, we define $g_{b_i}:\partial \Omega\setminus \{b_i\}\to \re$ as
\begin{equation*}
g_{b_i}(x):=\begin{cases}
\displaystyle\frac1n\int_{\gamma_b^x}f\, \de t\quad & \hbox{if } b \preceq x \prec b_i, 
\\
\displaystyle\frac1n\int_{\gamma_b^x}f\, \de t - 2\pi \quad & \hbox{if } b_i\prec x \prec b.
\end{cases}
\end{equation*}
Moreover, we introduce 
\begin{equation}\label{sum-gi}
g(x):= \sum_{i=1}^\ell m_i g_{b_i}(x)= \int_{\gamma_b^x} f(y)\, \de y - 2\pi
\sum_{i=0}^{j-1} m_i \qquad \hbox{if } b_{j-1} \prec x \prec b_{j}, 
\end{equation}
with $j\in \{1,\ldots,\ell+1\}$, $b_0=b_{\ell+1}:=b$, and $m_0:=0$. Note that, since $g_{b_i}$ is continuous except at $b_i$, where it has a jump of $2m_i\pi$, we have that the boundary datum $g$ has a jump of $2\pi m_i$ 
at every $b_i$ (whereas it is continuous at $b$).

In view of the construction above, for every $\e>0$, every connected component of $\Om_\e(\ba)\setminus (\Sigma_1\cup\ldots\cup \Sigma_n)$ is simply connected. 
Therefore, the energy \eqref{energy-n} can be expressed as 
\begin{equation}\label{recast-1}
\mathcal E_\eps^{(n)}(\ba)=\frac12\int_{\Omega_\eps(\ba)} |\nabla u^\e_{\ba}|^2\, \de x,  
\end{equation}
where $u^\e_{\ba}\in H^1(\Omega_\eps(\ba))$ is characterized by
\begin{equation}\label{euler-n}
\begin{cases}
\Delta u^\e_{\ba} = 0 & \hbox{in }\Omega_\eps(\ba)\setminus(\cup_{i=1}^n \Sigma_i),
\\
[u^\e_{\ba}]= 2\pi m_i \ &\hbox{on }\Sigma_i \cap \Omega_\eps(\ba),
\\
u^\e_{\ba} = \sum_{i=1}^\ell m_ig_{b_i}=
g & \hbox{on } \partial \Om\setminus \cup_{i=1}^n \overline{B}_\e(\a_i),
\\
{\partial u^\e_{\ba}}/{\partial \nu} = 0 \ &  \hbox{on }\partial \Om_\e(\ba)\setminus \partial\Omega,
\\
\partial (u^\e_{\ba})^+/\partial \nu = \partial (u^\e_{\ba})^-/\partial \nu & \hbox{on } (\cup_{i=1}^n \Sigma_i) \cap \Omega_\eps(\ba).
\end{cases}
\end{equation}
We point out that the expression \eqref{recast-1} is consistent with the definition of the energy in \eqref{energy-n}; in particular, when $a_i\equiv a$ for every $i=1,\ldots,n$, choosing $b_i\equiv b$, the datum $g$ in \eqref{sum-gi} is given by $g(x)=\int_{\gamma_\b^x} f(y)\, \de y$,
and we have  $\mathcal E_\e^{(n)}(a,\ldots,a)= n^2 \mathcal E_\e^{(1)}(a)$. 
Therefore, by \eqref{Fepsn}, 
for every $\epsilon$ small enough 
\begin{equation}\label{ecco}
\mathcal F_\eps^{(n)}(\underbrace{\a,\dots,\a}_{n\text{-times}})=n^2 \mathcal F_\epsilon^{(1)}(a)+n(n-1)\pi|\log \eps|.
\end{equation}

The following lemma shows that the construction of the $\Sigma_i$'s described above is arbitrary.
\begin{lemma}\label{lem-data}
The functional $\mathcal{E}_\e^{(n)}$ does not depend on the choice of the discontinuity points $\b_i$'s, nor the primitive $g$ of $f$, nor the cuts $\Sigma_i$'s.
\end{lemma}
\begin{proof}
It is enough to prove the result for $n=1$.
Let $\Sigma$ be a simple smooth curve in $\overline\Om$ connecting $a$ with $b$, intersecting $\partial B_\e(\a)$ at a single point and $\partial\Om$ only at $b$. 
Notice that each connected component of $\oea\setminus\Sigma$ is simply connected.
The condition $\curl F=0$ in $\oea$ implies that there exists $u\in H^1(\oea\setminus \Sigma)$ such that $ F=\nabla u$ in $\oea\setminus\Sigma$.
Therefore, the energy \eqref{recast-1} can be expressed as
\begin{equation}\label{energy3}
\mathcal{E}_\e^{(1)}(a)=\min\Bigg\{\frac{1}{2}\int_{\oea}|\nabla u|^2\de x\, : \,\text{$u \in H^1(\oea\setminus \Sigma)$, $u=  g$ on $\partial \Om\setminus \overline{B}_\e(\a)$}\Bigg\}. 
\end{equation}
To show that the energy above does not depend on the choice of the discontinuity point $\b$, nor on the primitive $g$ of $f$, nor on the curve $\Sigma$, we argue as follows.

Let $g$ and $g'$ be two primitives of $f$ with the same discontinuity point. Then the two boundary data differ by a constant. The same holds true for the corresponding minimizers of \eqref{energy3}, which have the same gradient and the same energy.

Let now assume that $g$ and $g'$ have two different discontinuity points, say $b$ and $b'$. Denote by $\Sigma$ and $\Sigma'$ the segments joining $\a$ and $\b$, and $\a$ and $b'$, respectively. Let $\phi$ and $\phi'$ be the angles associated with the discontinuity points $\b$ and $\b'$, respectively, in the angular coordinate centered at $\a$. It is not restrictive to assume that $\phi<\phi'$.
Let $u$ be the solution to \eqref{euler-u} associated with $g$, $\b$, and $\Sigma$; 
define $v:= u + 2 \pi \chi_S$, where $S$ is the subset of $\oea$ in which $\theta_\a \in (\phi,\phi')$. It is easy to see that $v$ is admissible for \eqref{energy3} associated with $g'$, $\b'$, and $\Sigma'$.
In particular, $v$ is the solution to \eqref{euler-u} associated with $g'$, $\b'$, and $\Sigma'$; moreover,
$$
\frac12\int_{\Omega_\e(\a)} |\nabla u|^2 \,\de x = \frac12\int_{\Omega_\e(\a)} |\nabla v|^2\,\de x.
$$
This concludes the proof of the invariance with respect to the discontinuity point in the case of a straight cut.

For general curves $\Sigma$ and $\Sigma'$, the region in $\oea$ lying between $\Sigma$ and $\Sigma'$ is the union of some simply connected sets $S_i$, $i\geq1$, bounded by portions of $\Sigma$ and $\Sigma'$.
In this case, the same strategy applies, provided one adds or subtracts $2\pi\chi_{S_i}$, according to whether the portions of $\Sigma$ bounding $S_i$ preceeds or follows that of $\Sigma'$ in the positive orientation of the angular coordinate.
This concludes the proof.
\end{proof}

Notice that the functional \eqref{recast-1} and the system \eqref{euler-n} are well defined even in the case that there exists $i\in\{1,\ldots,n\}$ such that $a_i\in\partial\Omega$.
In this case, Lemma \ref{lem-data} allows us to choose the discontinuity point $b_i=a_i$ and the cut $\Sigma_i=\emptyset$.

We may write $\nabla  u^\e_{\ba}$ as the sum of a singular and a regular part
\begin{equation}\label{recast-2}
\nabla  u^\e_{\ba} = \sum_{i=1}^n \omega(a_i) \k_{a_i} +\nabla \barv_{\ba},
\end{equation}
where the $\k_{a_i}$ satisfy \eqref{euler-k} and $\omega$ is defined in \eqref{omeghina}.
Here we impose that the angular coordinate $\theta_{a_i}$ centered at $a_i$ which defines $\k_{a_i}$ has a jump of $2\pi$ on the half line containing $\Sigma_i$ if $a_i\in \Om$, and on the half line containing the outer normal to the boundary at $a_i$ in case $a_i\in \partial \Omega$. 
The function $\barv_{\ba}\in H^1(\Om_\e(\ba))$ solves 
\begin{equation}\label{zucco}
  \begin{cases}
  \Delta \barv_{\ba}=0&\text{ in  }\Om_\eps(\ba),\\
	\barv_{\ba}=\sum_{i=1}^n \big(g_{b_i} - \omega(a_i) \theta_{a_i}\big)&\text{ on  }\partial\Omega\setminus\cup_{i=1}^n \overline{B}_{\e}(a_i), \\
{\partial \barv_{\ba}}/{\partial \nu}=-\sum_{i=1}^n \k_{a_i}\cdot \nu&\text{ on  } \partial \Om_\e(\ba)\setminus \partial \Om.
 \end{cases}
\end{equation}
The weight $\omega(a_i)$ is necessary in order to balance the jump of $g_{b_i}$ at $b_i$.
Using \eqref{recast-2}, we can write the energy \eqref{recast-1} as 
\begin{equation*}
\mathcal E_\eps^{(n)}(\ba)=\frac12\int_{\Omega_\eps(\ba)} \Big|\sum_{i=1}^n \omega(a_i) \k_{\a_i}+\nabla \barv_{\ba}\Big|^2\de x.
\end{equation*}

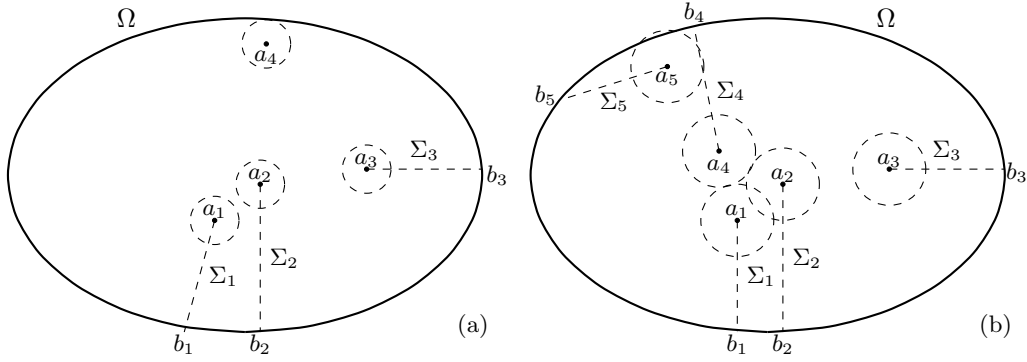
\begin{figure}
\begin{tikzpicture}[scale=0.4]

\draw[thick, domain=-3.141: 3.141, smooth, scale=1.3]plot({6*sin(\x r)}, {4*cos(\x r)}) node at (-3,4) {$\Omega$};

\coordinate (a1) at (-1,-1.5) node at (-1,-1.2) {\small $a_1$};
\coordinate (a2) at (0.5,-0.3) node at (0.5,0) {\small $a_2$};
\coordinate (a3) at (4,0.2) node at (4,0.5) {\small $a_3$};
\coordinate (a4) at (0.7,4.35) node at (0.7,3.9) {\small $a_4$};
\coordinate (b1) at (-2,-5.2) node at (-2,-5.6) {\small $b_1$};
\coordinate (b2) at (0.5,-5.2) node at (0.5,-5.6) {\small $b_2$};
\coordinate (b3) at (7.8,0.2) node at (8.3,0) {\small $b_3$};

\draw [fill=black] (a1) circle (2pt);
\draw [fill=black] (a2) circle (2pt);
\draw [fill=black] (a3) circle (2pt);
\draw [fill=black] (a4) circle (2pt);

\draw [dashed] (a1) circle (0.8); 
\draw [dashed] (a2) circle (0.8); 
\draw [dashed] (a3) circle (0.8); 
\draw [dashed] (a4) circle (0.8); 

\draw [dashed] (a1) --node[right] {\small $\Sigma_1$} (b1);
\draw [dashed] (a2) --node[right] {\small $\Sigma_2$} (b2);
\draw [dashed] (a3) --node[above] {\small $\Sigma_3$} (b3);

\coordinate (L) at (7.5,-5) node at (7.5,-5) {\small (a)};

\end{tikzpicture} 
\begin{tikzpicture}[scale=0.4]
\draw[thick, domain=-3.141: 3.141, smooth, scale=1.3]plot({6*sin(\x r)}, {4*cos(\x r)}) node at (3,4) {$\Omega$};

\coordinate (a1) at (-1,-1.5) node at (-1,-1.2) {\small $a_1$};
\coordinate (a2) at (0.5,-0.3) node at (0.5,0) {\small $a_2$};
\coordinate (a3) at (4,0.2) node at (4,0.5) {\small $a_3$};
\coordinate (a4) at (-3.3,3.6) node at (-3.3,3.2) {\small $a_5$};
\coordinate (a5) at (-1.6,0.8) node at (-1.6,0.3) {\small $a_4$};
\coordinate (b1) at (-1,-5.2) node at (-1,-5.6) {\small $b_1$};
\coordinate (b2) at (0.5,-5.2) node at (0.5,-5.6) {\small $b_2$};
\coordinate (b3) at (7.8,0.2) node at (8.2,0.1) {\small $b_3$};
\coordinate (b4) at (-6.8,2.5) node at (-7.3,2.65) {\small $b_5$};
\coordinate (b5) at (-2.4,4.8) node at (-2.4,5.4) {\small $b_4$};

\draw [fill=black] (a1) circle (2pt);
\draw [fill=black] (a2) circle (2pt);
\draw [fill=black] (a3) circle (2pt);
\draw [fill=black] (a4) circle (2pt);
\draw [fill=black] (a5) circle (2pt);

\draw [dashed] (a1) circle (1.2);
\draw [dashed] (a2) circle (1.2);
\draw [dashed] (a3) circle (1.2);
\draw [dashed] (a4) circle (1.2);
\draw [dashed] (a5) circle (1.2);

\draw [dashed] (a1) --node[right] {\small $\Sigma_1$} (b1);
\draw [dashed] (a2) --node[right] {\small $\Sigma_2$} (b2);
\draw [dashed] (a3) --node[above] {\small $\Sigma_3$} (b3);
\draw [dashed] (a4) --node[below] {\small $\Sigma_5$} (b4);
\draw [dashed] (a5) --node[right] {\small $\Sigma_4$} (b5);

\coordinate (L) at (7.5,-5) node at (7.5,-5) {\small (b)};
\end{tikzpicture}
\caption{(a) A crystal with 4 dislocations. (b) A crystal with 5 dislocations.}\label{fig-2}
\end{figure}

\begin{remark}\label{remuq}
By linearity,  the function $\bar{v}^\e_{\ba}$ introduced in \eqref{zucco} can be written as the superposition $\bar v^\e_{\ba}=\sum_{i=1}^n\bar v^\e_{a_i}$ of the solutions $\bar v_{a_i}^\e$ to 
\begin{equation*}
  \begin{cases}
  \Delta \bar v_{a_i}^\e=0&\text{ in  }\Om_\eps(\ba),\\
	\bar v_{a_i}^\e=g_{b_i}-\omega(a_i) \theta_{a_i}&\text{ on  }\partial\Omega\setminus\cup_{j=1}^n \overline{B}_{\e}(a_j), \\
{\partial \bar v_{a_i}^\e}/{\partial \nu}=- \k_{a_i}\cdot \nu&\text{ on  }\partial \Om_\e(\ba)\setminus \partial \Om.
 \end{cases}
\end{equation*}
Fix $i\in\{1,\ldots,n\}$. By linearity, the solution $\bar v^\e_{a_{i}}$ of the system above can be written as a superposition of one function satisfying a homogeneous Neumann boundary condition and another function satisfying a homogeneous Dirichlet boundary condition, namely  $\bar v^\e_{a_{i}}=   \hat{u}_{a_{i}}^\e + q_{a_{i}}^\e$, where $\hat{u}_{a_{i}}^\e$ and $q_{a_{i}}^\e$ solve, respectively, 
\begin{equation*}
\begin{cases}
\Delta \hat u_{a_{i}}^\e=0&\text{in $\Om_\e(\ba)$,}\\
	\hat u_{a_i}^\e=g_{b_i}- \omega(a_i)\theta_{a_i}&\text{on  }\partial\Omega\setminus\cup_{j=1}^n \overline{B}_{\e}(a_j),	\\
	{\partial \hat u_{a_{i}}}/{\partial \nu}=0 & \text{on  }\cup_{j=1}^n\partial B_\e(\a_j),
\end{cases}
\;\;
\begin{cases}
\Delta q_{a_{i}}^\e=0 & \text{in  }\Om_\e(\ba),\\
	q_{a_{i}}^\e=0 &\text{on  }\partial\Omega\setminus\cup_{j=1}^n \overline{B}_{\e}(a_j),	\\
	{\partial q_{a_{i}}}/{\partial \nu}=-\k_{a_i}\cdot \nu & \text{on  }\cup_{j=1}^n\partial B_\e(\a_j), \\
	 \oint_{\partial B_\e(a_j)}\partial q_{a_{i}}^\e/\partial \tau = 0 & \text{for every $j=1,\ldots,n$}.
\end{cases}
\end{equation*}
\end{remark}

\section{Auxiliary lemmas}\label{aux}
We start by proving some useful properties of the function $\k$ satisfying \eqref{euler-k}.
\begin{lemma}[Properties of $\k$]\label{scalare} 
Let $y_1$, $y_2\in \R^2$. Set $r:=|y_1-y_2|/2$ and $y$ the midpoint $(y_1+y_2)/2$. Then the vector fields $\k_{y_1}$ and $\k_{y_2}$ satisfy the following properties:
\begin{itemize}
\item[(i)] the scalar product $\k_{y_1}\cdot \k_{y_2}$ is negative in the disk $B_r(y)$, is zero on $\partial B_r(y)\setminus\{y_1,y_2\}$, and is positive in $\R^2\setminus \overline{B}_r(y)$.
In particular, 
$
\int_{B_r(y)} \k_{y_1}\cdot \k_{y_2}\, \de x \geq -2\pi;
$
\item[(ii)] the following estimate holds:
$\int_{B_r(y_i)} \k_{y_1}\cdot \k_{y_2}\, \de x \leq 2\pi,\quad\text{for $i=1,2$.}$
\end{itemize} 
Moreover,  let $\ell\in\mathbb N$, and let $y_1,\ldots, y_\ell\in \overline\Omega$ be distinct points. 
Then, 
\begin{itemize} 
\item[(iii)] if $y_i\in\Omega$, for every $0<\epsilon<d_i$ (see \eqref{dconi}),
\begin{equation}\label{solok}
2\pi|\log\epsilon|+2\pi \log d_i\leq\int_{\Omega_\epsilon(y_1,\dots,y_\ell)} |K_{y_i}|^2\, \de x\leq 2\pi|\log\epsilon|+2\pi\log(\mathrm{diam}\, \Omega);
\end{equation}
\item[(iv)] if $y_i\in\partial\Omega$, for every $0<\epsilon<\min_{j\neq i} \{\ea, |y_i-y_j|/2\}$, 
\begin{equation}\label{solok1}
\alpha|\log\e|+\alpha\log\ea\leq\int_{\Omega_\epsilon(y_1,\dots,y_\ell)} |K_{y_i}|^2\, \de x\leq \pi|\log\epsilon|+\pi\log(\mathrm{diam}\, \Omega).
\end{equation}
\end{itemize}
\end{lemma}
\begin{proof}
To prove (i) it is enough to notice that, given a point $x\in \R^2\setminus\{y_1,y_2\}$, the angle $\widehat{y_1\,x\,y_2}$ is larger than, equal to, or smaller than $\pi/2$, according to whether $x\in B_r(y)$, $x\in \partial B_r(y)$, $x\in \R^2\setminus \overline{B}_r(y)$, respectively. 
To prove the estimate, we consider the disk $B_r(y)$ whose diameter is the axis of the segment joining $y_1$ and $y_2$, and we define $B_r(y)^{+}$ the half of $B_r(y)$ on the side of $y_2$.
By symmetry, we have
\begin{equation*}
\begin{split}
\int_{B_r(y)} \k_{y_1}\cdot \k_{y_2}\, \de x =& -2\int_{B_r(y)^{+}} |\k_{y_1}\cdot \k_{y_2}|\, \de x\geq -\frac2r \int_{B_r(y)^{+}} |\k_{y_2}|\, \de x \\
\geq&-\frac2r\int_{\pi/2}^{3\pi/2}\int_0^r \de\rho_2\de\theta_2=-2\pi.
\end{split}
\end{equation*}

To prove (ii), observe that $|\k_{y_i}|\leq1/r$ in $\R^2\setminus\overline B_{r}(y_i)$, so that
$$\int_{B_r(y_i)} \k_{y_1}\cdot \k_{y_2}\, \de x \leq\frac{1}{r}\int_{B_{r}(0)} {|x|}^{-1}\,\de x=\pi.$$

To prove (iii) and (iv) one integrates in polar coordinates centered at $y_i$ over the sets $A_\e^{\mathrm{diam}\, \Om}(y_i)\cap\Om$, $A_\e^{d_i}(y_i)\cap\Om$, $A_\e^{s}(y_i)\cap \Om$ with $s:=\min_{j\neq i} \{\ea, |y_i-y_j|/2\}$, and uses the convexity assumption \eqref{H1} and \eqref{H3}.
\end{proof}

In the rest of this section, we prove some \emph{a priori} bounds on harmonic functions that will be useful in the sequel.
\begin{lemma}\label{max_princ_n}
Let $a_1,\dots,\a_n\in \overline \Om$, let $h\in W^{1,1}(\partial\Omega)$, 
and let $\hat u^\e$ be the solution to
\begin{equation}\label{max_equations}
 \begin{cases}
\Delta \hat u^\e=0&\text{in  }\Om_\e(\ba),\\
	\hat u^\e=h&\text{on  }\partial\Omega\setminus\cup_{j=1}^n \overline{B}_\e(\a_j),	\\
	{\partial \hat u^\e}/{\partial \nu}=0 & \text{on  }\partial\Om_\eps(\ba)\cap\big(\cup_{j=1}^n\partial B_\e(\a_j)\big).
\end{cases}
\end{equation}
Then $\hat u^\e$ is the minimizer of the Dirichlet energy in $H^1(\Om_\e(\ba))$ with prescribed boundary datum $h$ on $\partial\Omega\setminus\cup_{j=1}^n \overline{B}_\e(\a_j)$, and the following estimates hold
\begin{equation}\label{grad-n}
|\hat u^\e(x)| \leq C\|h\|_{L^\infty(\partial\Omega)}, \qquad\text{for all $x\in\overline\Om_\e(\ba)$}
\end{equation}
and 
\begin{equation}\label{grad-n2}
\int_{\Omega_\e(\ba)}|\nabla \hat u^\e|^2\,\de x\leq  C\|h\|^2_{H^{1/2}(\partial\Om)}
\end{equation}
for some constant $C>0$ independent of $\epsilon$. Moreover, in the case $a_1,\dots,\a_n\in \Om$ are distinct points and $\eps<\min_i d_i$, also the estimate on the gradient holds
\begin{equation}\label{grad-n3}
|\nabla \hat u^\e(x)| \leq C\|h\|_{L^\infty(\partial\Omega)}, \qquad\text{for all $x\in\cup_{j=1}^n\partial B_\e(\a_j)$}.
\end{equation}
\end{lemma}
\begin{proof}
The minimality of $\hat u^\e$ follows from the uniqueness of the solution to \eqref{max_equations}; estimate \eqref{grad-n} is a consequence of the maximum principle and Hopf's Lemma, while \eqref{grad-n2} is obtained by testing $\hat u^\e$ with the minimizer of the Dirichlet energy in $H^1(\Om)$ with prescribed boundary datum $h$ on $\partial\Omega\setminus\cup_{j=1}^n \overline{B}_\e(\a_j)$ and using the continuity of the map which associates the minimizer with the boundary datum. 
Finally, the proof of \eqref{grad-n3} is standard (see, e.g., \cite{evans}).
\end{proof}

As a simple consequence of Lemma~\ref{max_princ_n}, we have the following estimate in the case $n=1$.
\begin{lemma}\label{lemmaduenove}
Let $\a\in \Om$ and let $\e\in(0,d(a)/2)$. 
Let $w^\e$ be the harmonic extension in $B_\e(a)$ of $\bar u^\e_\a$.
Then there exists a constant $C>0$ independent of $\eps$ such that 
\begin{equation}\label{ce}
\int_{B_\e(a)}|\nabla w^\eps(x)|^2\,\de x\leq C\e^2\|g-\theta_\a\|^2_{L^\infty(\partial\Om)}.
\end{equation}
\end{lemma}
\begin{proof}
By applying Lemma \ref{max_princ_n} with $n=1$, we obtain that $\bar u_{\a}^\e\in W^{1,\infty}(\partial B_\e(a))$ and 
\begin{equation}\label{larichiamiamo}
\|\bar u_a^\e\|_{W^{1,\infty}(\partial B_\e(\a))}\leq C\|g-\theta_{\a}\|_{L^\infty(\partial\Om)}.
\end{equation}
It is a known fact in the theory of harmonic functions (see \cite{evans}) that there exists a constant $C>0$ such that for every harmonic function $\varphi\in H^{1}(B_1(0))$
$$ \|\nabla\varphi\|^2_{L^2(B_1(0))}\leq C\|\varphi-m(\varphi)\|^2_{H^{1/2}(\partial B_1(0))},$$
where $m(\varphi)$ is the average of $\varphi$ on $\partial B_1(0)$. 
Using the Poincar\'e inequality we have 
\begin{equation*}
\begin{split}
\|\varphi- m(\varphi)\|^2_{H^{1/2}(\partial B_1(0))} 
\leq & C\bigg(\|\varphi-m(\varphi)\|^2_{L^2(\partial B_1(0))}+\int_{\partial B_1(0)}\int_{\partial B_1(0)}\frac{|\varphi(x)-\varphi(y)|^2}{|x-y|^2}\,\de x\de y\bigg) \\
\leq &C
\bigg(\int_{\partial B_1(0)}\bigg|\frac{\partial\varphi(x)}{\partial\tau}\bigg|^2\,\de x+\int_{\partial B_1(0)}\int_{\partial B_1(0)}\frac{|\varphi(x)-\varphi(y)|^2}{|x-y|^2}\,\de x\de y\bigg).
\end{split}
\end{equation*}
Therefore, using this and employing the change of variables $x=\e x'+\a$ and $y=\e y'+\a$ we have
\begin{equation*}
\begin{split}
 \|\nabla w^\e & \|^2_{L^2(B_\e(\a))}=  \int_{B_1(0)}|\nabla w^\e(\e x'+\a)|^2\,\de x'  \\
 \leq  &
 C\bigg(\int_{\partial B_1(0)}\bigg|\frac{\partial w^\e(\e x'+\a)}{\partial\tau}\bigg|^2\,\de x'+\int_{\partial B_1(0)}\int_{\partial B_1(0)}\frac{|w^\e(\e x'+\a)-w^\e(\e y'+\a)|^2}{|x'-y'|^2}\,\de x'\de y'\bigg) \\
 =&
 C\bigg(\eps\int_{\partial B_\e(\a)}\bigg|\frac{\partial w^\e(x)}{\partial\tau}\bigg|^2\,\de x +\int_{\partial B_\e(\a)}\int_{\partial B_\e(\a)}\frac{|w^\e(x)-w^\e(y)|^2}{|x-y|^2}\,\de x\de y\bigg).
\end{split}
\end{equation*}
Owing to the fact that $w^\e=\bar u^\e$ on $\partial B_\e(a)$ and by \eqref{larichiamiamo}, we obtain \eqref{ce}.
\end{proof}
The following result is similar to  \cite[Lemma~I.5]{BBH}; however, for the sake of completeness, we present here a proof.
\begin{lemma}\label{lemma7giugno}
Let $D_1,\ldots,D_\ell \subset \R^2$ be open, bounded, simply connected sets with Lipschitz boundary, such that $\overline{D}_1,\ldots,\overline{D}_\ell$ are pairwise disjoint and intersect $\overline \Om$. 
Set $\Om'\coloneqq\Om\setminus \cup_{i=1}^\ell \overline D_i$ and for every $i=1,\ldots,\ell$, let $h_i\in C^1(\overline\Om{}')$.
Consider the minimum problem
 \begin{align}\label{minmin}
\min_{c}\min_{u}\frac{1}{2}\int_{\Om'}|\nabla u|^2\de x,
 \end{align}
with $c=(c_1,\dots,c_\ell)\in\R^\ell$ and $u$ varying in the subset of functions in $H^1(\Om')$ satisfying
\begin{equation}\label{adm-fun}
u = h_i + c_i \ \hbox{on }\partial D_i \cap \Om, \  \hbox{for }i=1,\ldots,\ell.  
\end{equation}
Then $p\in H^1(\Om')$ is a solution to \eqref{minmin} if and only if it satisfies
 \begin{equation}\label{7giugno}
 \begin{cases}
\Delta  p=0&\text{ in  }\Om', \\
{\partial p}/{\partial \nu}=0&\text{ on  }\partial\Omega' \cap \partial \Om ,\\
{\partial p}/{\partial \tau}={\partial h_i}/{\partial \tau}&\text{ on  }\partial D_i\cap\Om,\ \hbox{for }i=1,\ldots,\ell,\\
{\int_{\partial D_i \cap \Om}{\partial p}/{\partial \nu}\,\de\mathcal H^1=0}&\text{ for }i=1,\dots,\ell.
\end{cases}  
\end{equation}
Moreover, a solution $p$ satisfies 
\begin{equation}\label{stima-p}
\big|\max_{\overline{\Om}'} p - \min_{\overline{\Om}'} p\big| \leq 
\sum_{i=1}^\ell \mathcal H^1(\partial D_i\cap\Om)\max_{\partial D_i \cap \Om}\bigg|\frac{\partial h_i(x)}{\partial \tau}\bigg|.
\end{equation}
\end{lemma}

\begin{proof}
First we remark that problem \eqref{minmin} is well posed, namely the class of admissible functions is not empty: indeed, the function
$u:=\sum_{i=1}^\ell\zeta_i h_i $ satisfies \eqref{adm-fun}, being $\zeta_i$ suitable smooth functions in $\Om'$ which are identically 1 in a neighborhood of $D_i$. 

If $p$ satisfies the minimum problem \eqref{minmin}, then it is easy to obtain system \eqref{7giugno} as the Euler-Lagrange conditions arising from minimality.

Conversely, assume that $p$ satisfies \eqref{7giugno}. In order to show that $p$ is optimal for \eqref{minmin}, it is enough to show that the \eqref{7giugno} has a unique 
 solution (up to a constant). Take two solutions of \eqref{7giugno}; then their difference $s$ is characterized by
\begin{align}\label{7giugno2}
 \begin{cases}
\Delta  s=0&\text{ in  }\Omega', \\
{\partial s}/{\partial \nu}=0&\text{ on  }\partial\Omega' \cap \partial \Om,\\
{\partial s}/{\partial \tau}=0&\text{ on  }\partial \Om' \setminus \partial \Om,\\
{\int_{\partial D_i\cap \Om}{\partial s}/{\partial \nu}\;\de\mathcal H^1=0}&\text{ for }i=1,\dots,\ell.
\end{cases}  
 \end{align}
The third condition implies that $s$ is constant on the boundary of the holes (intersected with $\Om$), namely $s=s_i$ on  $\partial D_i\cap \Omega$  for some $s_i\in \R$, for $i=1,\ldots,\ell$. 
In view of the maximum principle, we infer that either $s$ attains its maximum on the boundary $\partial \Om'\setminus \partial\Om$ or it is constant. 
The former case is excluded by Hopf's Lemma combined with the second and fourth conditions in \eqref{7giugno2}; therefore the latter case holds and the proof of the first statement is concluded. 

We relabel the sets $D_i$ so that the local minima of $p$ on their boundary are ordered as follows:
$$
\min_{\partial D_1 \cap \Om} p \geq \min_{\partial D_2 \cap \Om} p \geq \ldots \geq \min_{\partial D_\ell \cap \Om} p = \min_{\overline{\Om}'} p.
$$
Notice that, for the last equality, we have used the maximum principle, combined with the Hopf's Lemma, which prevents the minimum of $p$ to be on $\partial \Om$ (the same holds true for the maximum). We now define an ordered subfamily of indices as follows: we choose
$$
j_1\in \Big\{ j\ :\ \max_{\partial D_j \cap \Om} p  = \max_{\overline{\Om}'} p \Big\},
$$
and, by recursion, given $j_i$ we choose
$$
j_{i+1}\in  \Big\{ j\ :\ \max_{\partial D_j \cap \Om} p = \max_{\cup_{k>{j_i}} \partial D_k \cap \Om} p \Big\}.
$$
Since $j_{i+1}>j_i$, in a finite number of steps, say $i^*$,  we find $j_{i^*}=\ell$ and the procedure stops. We clearly have
$$
\max_{\overline{\Om'}} p = \max_{\partial D_{j_1} \cap \Om} p \geq \max_{\partial D_{j_2} \cap \Om} p \geq \ldots \geq \max_{\partial D_{j_{i^*}}\cap \Om} p =  \max_{\partial D_\ell \cap \Om} p.
$$
Moreover, we claim that, for every $i=1,\ldots,i^*-1$
\begin{equation}\label{maxmin}
\min_{\partial D_{j_i} \cap \Om} p \leq \max_{\partial D_{j_{i+1}}\cap \Om}p. 
\end{equation}
Once proved the claim we are done, indeed
\[
\begin{split}
&\big|\max_{\overline{\Om}'}  - \min_{\overline{\Om}'} p\big| 
= \max_{ \partial D_{j_1}\cap \Om} p - \min_{\partial D_{j_{i^*}}\cap \Om} p 
\\ & =  \sum_{i=1}^{i^*} ( \max_{ \partial D_{j_i}\cap \Om} p - \min_{\partial D_{j_i}\cap \Om} p ) + \sum_{i=1}^{i^*-1} (\min_{\partial D_{j_i}\cap \Om} p - \max_{\partial D_{j_{i+1}}\cap \Om} p)
\\
& \leq  \sum_{i=1}^\ell( \max_{ \partial D_i\cap \Om} p - \min_{\partial D_i\cap \Om} p )\leq \sum_{i=1}^\ell \mathcal H^1(\partial D_i\cap\Om)\max_{\partial D_i\cap \Om}\bigg|\frac{\partial h_i(x)}{\partial \tau}\bigg|,
\end{split}
\]
where, in the last line, we have used the Mean Value Theorem. 
This proves \eqref{stima-p}.

In order to prove the claim \eqref{maxmin}, we argue by contradiction: we assume that there exists $i\in \{1,\ldots,i^*-1\}$ such that
$$
\underline{p}:=\min_{\partial D_{j_i}\cap \Om} p > \overline{p}:=\max_{\partial D_{j_{i+1}}\cap \Om} p.
$$
Notice that, by construction, we have $$\min_{\cup_{k=1}^{j_i}\partial D_{k}\cap \Om} p > \max_{\cup_{k= {j_{i+1}}}^\ell \partial D_{k}\cap \Om} p.$$ 
Consider the nonempty set $E:=\{x\in \Om' :\quad \underline{p}>p(x)>\overline{p}\}$
and the function
\[
\psi(x):=
\begin{cases}
p(x) &\text{if  $p(x)\leq \overline{p}$},\\
\overline{p}  &\text{if $x\in E$},\\
p(x)+\overline{p} -\underline{p} &\text{if $p(x)\geq \underline{p}$}.
\end{cases}
\]
The function $\psi$ is an element of $H^1(\Om')$, it is admissible for the minimum problem \eqref{minmin}, and it decreases the Dirichlet energy, contradicting the optimality of $p$. 
\end{proof}

As a direct consequence of Lemma \ref{lemma7giugno} we obtain the following corollary.

\begin{corollary}
Let $a_1,\dots,\a_\ell\in \overline \Om$ be distinct points and let $(\a_1^\e,\dots,\a_\ell^\e)$ be a sequence of points in $\Om^\ell$ converging to $(\a_1,\dots,\a_\ell)$ as $\e\to0$, such that the disks $B_\e(a_j^\e)$ are pairwise disjoint. Fix $i\in\{1,\dots,\ell\}$ and let $p_i^\e$ be the solution (unique up to a constant) to
\begin{equation}\label{5ottobre}
\begin{cases}
\Delta p_i^\e=0 & \text{in  }\Om_\e(a_1^\e,\ldots,a_\ell^\e),\\
	\partial p_i^\e/\partial \nu=0 & \text{on  }\partial\Omega\setminus\cup_{j=1}^\ell \overline B_\e(a_j^\e),	\\
	{\partial p_i^\e}/{\partial \tau}=-\partial\log(|x-a_i^\e|)/\partial\tau & \text{on  }\partial B_\e(\a_j^\e)\cap\Omega, \text{ for every $j\neq i$,}\\
	{\partial p_i^\e}/{\partial \tau}=0 & \text{on  } \partial B_\e(\a_i^\e)\cap \Omega,\\
	\int_{\partial B_\e(a_j^\e)\cap\Om}\partial p_i^\e/\partial \nu = 0 & \text{for } j=1,\ldots,\ell.
\end{cases}
\end{equation}
Then, there exists a positive constant $C$ independent of $\e$ such that 
\begin{equation}\label{latesi1}
||p_i^\e||_{L^\infty(\Om_\e(a_1^\e,\ldots,a_\ell^\e))}\leq C\e.
\end{equation}
Moreover,
\begin{equation}\label{latesi}
\int_{\Om_\e(a_1^\e,\ldots,a_\ell^\e)}|\nabla p_i^\e|^2\, \de x 
\leq \frac{C}{{|\log\e|}}.
\end{equation}
\end{corollary}
\begin{proof}
By applying Lemma \ref{lemma7giugno} with $D_j=B_\e(a_j^\e)$, and $h_j=-\log(|x-a_i^\e|)$ for every $j\in\{1,\ldots,\ell\}\setminus\{i\}$, $h_i=0$, estimate \eqref{stima-p} provides a positive constant $C$ such that
\begin{equation*}
\Big|\max_{\Om_\e(a_1^\e,\ldots,a_\ell^\e)} p_i^\e - \min_{\Om_\e(a_1^\e,\ldots,a_\ell^\e)} p_i^\e \Big| \leq 2\pi\e C.
\end{equation*}
which implies \eqref{latesi1}. 

Let $\zeta_\e$ be a smooth cut-off function which is identically 1 inside $B_{\e}(0)$ and  0 outside $B_{\sqrt\e}(0)$, and whose support vanishes as $\e\to0$. 
From the minimality of $p_i^\e$, comparing with $c_1=\cdots=c_\ell=0$ and 
$$u(x)=-\sum_{j\neq i} \zeta_\e (x-a_j^\e) \log(|x-a_i^\e|)$$
in \eqref{adm-fun}, we obtain
\begin{equation}\label{capacity}
\int_{\Om_\e(a_1^\e,\ldots,a_\ell^\e)} |\nabla p_i^\e|^2\, \de x  \leq C\bigg( \int_{B_{\sqrt\e}(0)} |\nabla \zeta_\e|^2\, \de x+|\spt\zeta_\e|\bigg),
\end{equation}
for some positive constant $C$ independent of $\e$.
By infimizing with respect to $\zeta_\e \in \{\zeta \in C^\infty_c(B_{\sqrt\e}(0))\,,\ \zeta \equiv 1 \ \hbox{in }B_{\e}(0)\}$, the integral in the right-hand side of \eqref{capacity} is the capacity of a disk of radius $\e$ inside a disk of radius $\sqrt\e$, which is equal to $4\pi/|\log\e|$, therefore we obtain
$$\int_{\Om_\e(a_1^\e,\ldots,a_\ell^\e)} |\nabla p_i^\e|^2\, \de x  \leq C\pi\bigg(\frac{4}{|\log\e|}+\e\bigg),$$
which implies \eqref{latesi}.
The lemma is proved.
\end{proof}

\begin{lemma}\label{conservative-n}
Let $a_1,\dots,a_n$ be distinct points in $\Om$ and let $(\a_1^\e,\dots,\a_n^\e)$ be a sequence of points in $\Om^n$ converging to $(\a_1,\dots,\a_n)$ as $\e\to0$.
Then, for every $i=1,\ldots,n$, as $\e\to 0$ we have
\begin{equation}\label{thesis-515}
\|\nabla\bar v^\e_{a_i^\e} - \nabla v_{a_i} \|_{L^2(\Omega_\e(\ba);\R^2)} \to 0,
\end{equation}
where $\bar v_{a_i^\e}^\e$ is defined in Remark~\ref{remuq} and $v_{a_i}$ satisfies
\begin{equation}\label{def-barv-a}
\begin{cases}
 \Delta  v_{\a_i}=0&\text{ in  }\Omega,\\
 v_{\a_i}=g_{b_i}-\theta_{\a_i}&\text{ on  }\partial\Omega.
\end{cases}
\end{equation}
\end{lemma}
\begin{proof}
Since the limit points $a_1,\ldots,a_n$ are all distinct and far from the boundary, we may assume, without loss of generality, that $\e<d:= \min_{i\in \{1,\ldots,n\}} d_i$ (see \eqref{dconi}), so that, for $\e$ small enough, the disks $B_\e(a_i^\e)$ are pairwise disjoint and do not intersect $\partial \Omega$. Let now $i\in\{1,\ldots,n\}$ be fixed. 
By Remark \ref{remuq}, the solution $\bar v^\e_{a_{i}^\e}$ can be written as $\bar v^\e_{a_{i}^\e}=\hat{u}_i^\e + q_i^\e$. 
Then to prove \eqref{thesis-515}, we will show that, for every fixed $i\in\{1,\ldots,n\}$,
\begin{subequations}
\begin{align}
& \| \nabla \hat u_i^\e - \nabla v_{a_i^\e}\|_{L^2(\Om_\e(\ba^\e);\R^2)} \to 0, \label{claim-uhat}
\\
&  \|\nabla q_i^\e \|_{L^2(\Om_\e(\ba^\e);\R^2)} \to 0. \label{claim-qeps}
\end{align}
\end{subequations}
Note that in \eqref{thesis-515} we can replace $v_{a_i}$ with $v_{a_i^\e}$ thanks to the continuity of the map  $C^0(\partial \Om)\ni h\mapsto w_h\in H^1(\Omega)$, where $w_h$ is harmonic in $\Om$ with boundary datum $h$.

By Lemma \ref{max_princ_n}, $\hat u_i^\e$ minimizes the Dirichlet energy in $H^1(\Om_\e(\ba^\e))$ with boundary datum $g_{b_i^\e}- \theta_{a_i^\e}$ on $\partial\Om$, and by \eqref{grad-n2} we have
\begin{equation*}
\|\nabla \hat u_i^\e \|^2_{L^2(\Om_\e(\ba^\e);\R^2)} \leq  C \|g_{b_i^\e}-\theta_{a_i^\e}\|^2_{H^{1/2}(\partial\Om)}.
\end{equation*}
Let us consider the extension of $\hat u_i^\e$ (not relabeled), which is harmonic in every $B_\e(a_j^\e)$. 
By \eqref{grad-n} and \eqref{grad-n3}, an estimate similar to \eqref{ce} holds: there exists $C>0$ independent of $\e$ such that
\begin{equation*}
\|\nabla \hat u_i^\e\|_{L^2(B_\e(a_j^\e);\R^2)}^2\leq C\e^2 \|g_{b_i^\e}-\theta_{a_i^\e}\|_{L^\infty(\partial \Om)}^2.
\end{equation*}
Since all the points $\a_i^\e$ are in $\Omega^{(d/2)}$, by the two estimates above,
we obtain a uniform bound in $H^1(\Om)$ for $\hat u_i^\e$, and therefore $\hat u_i^\e$ has a subsequence that converges weakly to a function $w_i$ in $H^1(\Om)$ as $\e\to0$. 
By the lower semicontinuity of the $H^1$ norm {
and the minimality of $\hat u_i^\e$ and $v_{a_i}$}, it turns out that $w_i=v_{a_i}$ and 
\[
\lim_{\e\to0} \|\nabla \hat u_i^\e\|_{L^2(\Omega;\R^2)}=\|\nabla  v_{a_i}\|_{L^2(\Omega;\R^2)};
\]
moreover, by the triangle inequality we have that
\begin{equation*}
\|\nabla \hat u_i^\e-\nabla v_{\a_i^\e}\|_{L^2(\Omega;\R^2)}\leq \|\nabla \hat u_i^\e-\nabla v_{\a_i}\|_{L^2(\Omega;\R^2)}+\|\nabla  v_{\a_i}-\nabla v_{\a_i^\e}\|_{L^2(\Omega;\R^2)},
\end{equation*}
with $v_{a_i}$ solution to \eqref{def-barv-a}.
The first term in the right-hand side above vanishes as $\e\to0$.
The second term converges to zero thanks to the continuity of the map that associates $v_a$ with the boundary datum $g-\theta_a$ (observe that $g-\theta_a^\eps\rightarrow g-\theta_a$ in $H^{1/2}(\partial\Om)$), so that 
\begin{equation}\label{strong-uhat}
\|\nabla \hat u_i^\e - \nabla v_{a_i^\e}\|_{L^2(\Omega;\R^2)}\to 0, 
\end{equation}
hence \eqref{claim-uhat} follows.

To obtain an estimate for the gradient of $q_i^\e$, we define $P_i^\e:=(\nabla q_i^\e)^\perp=(-\partial_2q_i^\e,\partial_1q_i^\e)$ and notice that $P_i^\e$ is a conservative vector field in $\Om_\e(\ba)$, namely its circulation vanishes along any closed loop contained in $\Om_\e(\ba^\e)$. 
Therefore, there exists $p_i^\e\in H^1(\Om_\e(\ba^\e))$ such that $P_i^\e=\nabla p_i^\e$ and, in view of Remark~\ref{remuq} with the fact that $-\k_{a_i^\e}\cdot\nu=\partial\log(|x-a_i^\e|)/\partial\tau$, $p_i^\e$ is a solution to \eqref{5ottobre}. Then by \eqref{latesi} we deduce that the gradient of $p_i^\e\to 0$ as $\e\to 0$ and since the gradients of $p_i^\e$ and $q_i^\e$ have the same modulus, we obtain \eqref{claim-qeps}. 
\end{proof}

\begin{remark}\label{rem-conv}
From the proof of Lemma \ref{conservative-n} we derive some important properties of the functions $\hat u_i^\e$ 
and $p_i^\e$ introduced in Remark~\ref{remuq}
and \eqref{5ottobre}, respectively. 
Let $a_1,\dots,a_n$ be distinct points in $\Om$ and let $(\a_1^\e,\dots,\a_n^\e)$ be a sequence of points in $\Om^n$ converging to $(\a_1,\dots,\a_n)$ as $\e\to0$.
Then, for every $i=1,\ldots,n$, as $\e\to 0$, we have
\begin{itemize}
\item[(i)] the function $\hat u_i^\e$ admits an extension which converges to $v_{a_i}$ strongly in $H^1(\Omega)$. This follows by combining the fact that $v_{a_i^\e}\to v_{a_i}$ strongly in $H^1(\Omega)$, $ \hat u_i^\e$  and $v_{a_i^\e}$ agree on the boundary, the Poincar\'e inequality, and property \eqref{strong-uhat};
\item[(ii)] the function $p_i^\e$ admits a (not relabeled) extension $p_i^\e\in H^1(\Om)$. 
In fact, recalling that $\phi_{a_i^\epsilon}(\cdot)=\log(|\cdot-a_i^\epsilon|)$, since $p_i^\e=\phi_{a_i^\epsilon}+c_j$ on $\partial B_\e(a_j^\e)$ for $j\neq i$ and $p_i^\e=c_i$ on $\partial B_\e(a_i^\e)$ for some constants $c_1,\dots, c_n$, we define $p_i^\e=\phi_{a_i^\e}+c_j$ on $B_\e(a_j^\e)$ and $p_i^\e\equiv c_i$ on $B_\e(a_i^\e)$. Thus, up to subtracting a constant,  thanks to \eqref{latesi1}, $p_i^\e$ is observed to tend to $0$ strongly in $H^1(\Om)$.
\end{itemize}
\end{remark}

\begin{lemma}\label{ext_lemma}
Let $a_1,\ldots,a_\ell$ be distinct points in $\overline\Omega$ and let $\bar v^\e_{a_i}\in H^1(\Om_\e(a_1,\ldots,a_\ell))$ solve the system in Remark~\ref{remuq} for some $i\in\{1,\ldots,\ell\}$.
Then the function $\bar v^\e_{a_i}$ admits an extension (not relabeled) in $H^1(\Omega)$.
In particular, there exist $\bar\e>0$ and a constant $C>0$ such that
\begin{equation*}
\int_{B_\e(a_j)}|\nabla \bar v^\eps_{a_i}|^2\,\de x\leq C,
\end{equation*}
for all $\eps< \bar\eps$ and for every $j\in\{1,\ldots,\ell\}$.
\end{lemma}
\begin{proof}
We consider the domain $E:=\R^2\setminus(\cup_{i=1}^\ell \overline{B}_{\e}(a_i)\cup\overline\Om)$ and we extend $\bar v^\e_{a_i}$ on $E$ by defining it as the solution to the Dirichlet problem with datum $g_{b_i}-\omega(a_i)\theta_{a_i}$ on $\partial E\cap\partial\Omega$.
It is easy to check (one can use, for instance, \eqref{grad-n2}) 
that there exists a constant $C>0$ such that 
$$\int_{E}|\nabla \bar v^\e_{a_i}|^2\,\de x\leq C.$$
Let $\bar\e:=\tfrac14\min\{|a_i-a_j| : i,j\in\{1,\ldots,\ell\},i\neq j\}$.
For $\e<\bar\e$, consider a family of functions $\zeta_\e\in C^\infty(B_\e(0))$ which are zero in a neighborhood of $B_{\e/2}(0)$, equal to $1$ in a neighborhood of $\partial B_\e(0)$, and such that $||\nabla\zeta_\e||_{L^2(B_\e(0))}$ is uniformly bounded. 
We shall exploit the fact that $\bar v^\e_{a_i}$ is defined in the annulus $A_\e^{2\e}(a_j)$ for every $j\in\{1,\ldots,\ell\}$ to define its extension in $B_\e(a_j)$.
To this aim, consider the inversion function $\mathbb I_{a_j}^\e\colon\mathbb C\setminus \{a_j\} \to \mathbb C\setminus \{a_j\}$ given by $\mathbb I^\e_{a_j} (x):=\e^2 (x-a_j)/|x-a_j|^2$, and define 
$$\bar v^\e_{a_i}(x):=
\begin{cases}
\zeta_\e(x)\bar v^\e_{a_i}(\mathbb I_{a_j}^\e(x)) & \text{if $x\in\mathbb I_{a_j}^\e(A_\e^{2\e}(a_j))$,} \\
0 & \text{elsewhere in $B_\e(a_j)$.}
\end{cases}$$
Also in this case, an easy check shows that 
$$\int_{B_{\e}(a_j)}|\nabla \bar v^\e_{a_i}|^2\,\de x\leq C\int_{A^{2\e}_\e(a_j)}|\nabla \bar v^\e_{a_i}|^2\,\de x\leq C,$$
for some constant $C>0$ independent of $\bar\eps$, and for all $\e<\bar\e$. 
The lemma is proved.
\end{proof}

\section{The limit as $\e \to 0$}\label{twomeglcheone}

This section is devoted to the proof of the main results, Theorem~\ref{thGcn} and Corollary~\ref{confinon}, which is presented in Subsection~\ref{nonhaunnome}. 
Our proof strategy will rely on the results in the case of one dislocation $a\in\overline\Omega$, which we treat next in Subsection~\ref{n1}.
To study the asymptotic behavior of the rescaled energies \eqref{Fepsn}, we distinguish two different scenarios (recall the function $d_i$ defined in \eqref{dconi}):
\begin{itemize}
\item[-] all the limit points are in the interior of $\Omega$ and are all distinct, namely $\min_i d_i>0$ (treated in Subsection~\ref{ngen1});
\item[-] either at least two limit points coincide or one limit point is on the boundary $\partial \Om$, namely $\min_i d_i=0$ (treated in Subsection~\ref{ngen2}).
\end{itemize}

\subsection{The case $n=1$}\label{n1}
Given $\a\in \overline \Om$, the energy $\E\coloneqq\mathcal{E}_\e^{(1)}(a)$ in \eqref{energy-n} reads
\begin{equation*}
\E =\min \bigg\{\frac{1}{2}\int_{_{\oea}} \!\!|F|^2\,\de x:  \text{$ F\in L^2(\oea;\mathbb R^2)$, $\curl  F=0$, $ F\cdot \tau=f$ on $\partial \Om\setminus \overline{B}_\e(a)$}\bigg\},
\end{equation*}
and its rescaling $\mathcal F_\eps(\a)\coloneqq\mathcal{F}_\e^{(1)}(a)$ in \eqref{Fepsn} reads
\begin{equation}\label{Feps}
\mathcal F_\eps(\a)=\frac12\int_{\Om_\e(\a)}|\nabla u_\a^\e|^2\,\de x-\pi|\log \eps|\,,
\end{equation}
where $u_\a^\e$ solution to \eqref{euler-n} with $n=1$, namely
\begin{equation}\label{euler-u}
\begin{cases}
\Delta u^\e_\a = 0 & \hbox{in } \oea\setminus\Sigma, \\
[u^\e_\a]= 2\pi\ &\hbox{on }\Sigma\cap \Omega_\e(\a), \\
u^\e_\a = g & \hbox{on } \partial \Om\setminus \overline{B}_\e(\a), \\
{\partial u^\e_\a}/{\partial \nu} = 0 \ &  \hbox{on }\partial B_\e(\a) \cap \Omega, \\
\partial (u^\e_\a)^+/\partial \nu = \partial (u^\e_\a)^-/\partial \nu & \hbox{on } \Sigma\cap \Omega_\e(\a).
\end{cases}
\end{equation}
In the last equation above, $\nu$ is a choice of the unit normal vector to $\Sigma$.
Notice that, when $\dist(\a,\partial \Omega)\leq\e$, the choice of $b\in \partial\Omega\cap B_\e(\a)$ implies that $\oea\setminus \Sigma =\oea$ and the jump condition of $u^\e_\a$ across $\Sigma$ is empty.

In this case, the decomposition \eqref{recast-2} reads
\begin{equation}\label{dec1}
 \nabla u^\e_\a=\omega(a)\k_\a+\nabla \bar u^\e_\a,
\end{equation}
where $\bar u^\e_\a\in H^1(\oea)$ is the solution to
\begin{align}\label{dec_eqn}
\begin{cases}
  \Delta \bar u^\e_\a=0 &\text{ in  }\oea, \\
	\bar u^\e_\a=g-\omega(a)\theta_\a &\text{ on  }\partial \Omega\setminus \overline{B}_\e(a), \\
{\partial \bar u^\e_\a}/{\partial \nu}=0&\text{ on  }\partial B_\e(\a).
 \end{cases}
\end{align}
Notice that this function is the same as the $\bar v_\ba^\e$ in Remark~\ref{remuq} when $n=1$, thanks to \eqref{euler-k}.

Given $a^\e\to a$ as $\e\to0$, the asymptotics of $\mathcal F_\eps(\a^\e)$ when $\a\in \Omega$ is dealt with in Proposition~\ref{prop1}, whereas the case in which $a\in \partial \Omega$ is dealt with in Proposition~\ref{prop2}.
These two results readily imply Theorem~\ref{thGcn} in the case $n=1$ (the continuity of the limit functional $\mathcal F$ follows from the relationship between continuous convergence and $\Gamma$-convergence, see \cite[Remark 4.9]{dalmaso}).
\begin{proposition}\label{prop1}
Let $\a\in \Om$ and $v_\a$ satisfy \eqref{def-barv-a}. 
For every sequence $\a^\eps\to\a$ as $\e\to0$ we have
\begin{equation}\label{52}
\mathcal F_\eps(\a^\eps)\rightarrow
\pi\log d(\a)+\frac{1}{2}\int_{\Om_{d(\a)}(\a)}|\k_\a + \nabla v_\a |^2\,\de x+\frac{1}{2}\int_{B_{d(\a)}(\a)}|\nabla v_\a|^2\,\de x.
\end{equation}
\end{proposition}
\begin{proof}
Since $a\in\Omega$, $\omega(a)=1$.
Since $d(a^\epsilon)\rightarrow d(a)$ and $d(a)>0$, we can take $\e$ so small so that $\e<\min\{d(a^\e),1\}$. 
By plugging \eqref{dec1} into \eqref{Feps}, $\mathcal F_\eps(\a^\eps)$ reads
$$\mathcal F_\eps(\a^\eps)=\frac{1}{2}\int_{\Om_\e(\a^\e)}|\nabla \bar u_{\a^\eps}^\e|^2\,\de x+\int_{\Om_\e(\a^\e)} \k_{\a^\eps}\cdot\nabla\bar u_{\a^\eps}^\e \,\de x +\frac{1}{2}\int_{\Om_{\eps}(\a^\e)}|\k_{\a^\eps}|^2\,\de x+\pi\log \eps.$$
Set, for brevity, $d:=d(\a)$, $d^\eps:=d(\a^\e)$, $\k:=\k_{\a}$, and $\k^\e:=\k_{\a^\e}$.
Writing 
\begin{equation*}
\begin{split}
\frac{1}{2}\int_{\Om_{\eps}(\a^\e)}|\k^{\e}|^2\,\de x=  \frac{1}{2}\int_{\Om_{d^\eps}(\a^\e)}|\k^{\e}|^2\,\de x+\frac{1}{2}\int_{A_{\eps}^{d^\e}(\a^\e)}|\k^{\e}|^2\,\de x = \frac{1}{2}\int_{\Om_{d^\eps}(\a^\e)}|\k^{\e}|^2\,\de x+\pi\log \frac{d^\e}\e,
\end{split}
\end{equation*}
we obtain
\begin{align}\label{formercase}
\mathcal F_\eps(\a^\e)= \pi\log d^\e+\frac{1}{2}\int_{\Om_{d^\e}(\a^\e)}|\k^{\e}|^2\,\de x+\int_{\Om_\e(\a^\e)} \k^{\e}\cdot\nabla\bar u_{\a^\e}^\e \,\de x + \frac{1}{2}\int_{\Om_\e(\a^\e)}|\nabla \bar u_{\a^\e}^\e|^2\,\de x.
\end{align}
If we prove that, as $\epsilon\to 0$,
\begin{subequations}
\begin{align}
\pi\log d^\e+\frac{1}{2}\int_{\Om_{d^\e}(\a^\e)}|\k^{\e}|^2\,\de x &\rightarrow \pi\log d+\frac{1}{2}\int_{\Om_{d}(\a)}|\k|^2\,\de x, \label{54a}\\
\int_{\Om_\e(\a^\e)} \k^{\e}\cdot\nabla\bar u_{\a^\e}^\e \,\de x &\rightarrow \int_{\Om_{d}(\a)} \k\cdot\nabla v_{\a}\,\de x, \label{claim2}\\
\frac{1}{2}\int_{\Om_\e(\a^\e)}|\nabla \bar u_{\a^\e}^\e|^2\,\de x &\rightarrow \frac{1}{2}\int_{\Om}|\nabla  v_{\a}|^2\,\de x,\label{claim1}
\end{align}
\end{subequations}
where $v_\a$ satisfies \eqref{def-barv-a}, then \eqref{52} follows.

Since $d^\eps\rightarrow d$, $\k^{\e}\chi_{B_{d^\eps}({\a^\e})}$ converges pointwise to $\k\chi_{B_{d}({\a})}$, and $\k^{\e}\chi_{B_{d^\eps}({\a^\e})}$ are uniformly bounded for $\eps$ small enough, \eqref{54a} follows by the Dominated Convergence Theorem.

To prove \eqref{claim2}, we integrate by parts to obtain
 $$\int_{\Om_\e(\a^\e)} \k^{\e}\cdot\nabla\bar u_{\a^\e}^\e \,\de x=\int_{\partial\Om} (\k^{\e}\cdot \nu) (g-\theta_{\a^\e})\,\de x,$$
where we have used \eqref{euler-k} and the fact that $\bar u_{\a^\e}^\e=g-\theta_{\a^\e}$ on $\partial\Om$ (see \eqref{dec_eqn}). 
Since $\k^{\e}$ and $\theta_{\a^\e}$ are uniformly bounded in $\eps$ on the set $ \partial\Om$, and converge pointwise to $\k$ and $\theta_{\a}$, respectively, by the Dominated Convergence Theorem, we have
$$
\int_{\Om_\e(\a^\e)} \k^{\e}\cdot\nabla\bar u_{\a^\e}^\e \,\de x \rightarrow \int_\Om \k\cdot\nabla v_{\a}\,\de x = \int_{\Om_{d}(\a)} \k\cdot\nabla v_{\a}\,\de x \,,
$$
which gives \eqref{claim2}. The last equality follows by the Divergence Theorem, combined with \eqref{euler-k}. 

It remains to prove \eqref{claim1}. 
To do this, consider the harmonic extension $w^\e$ of $\bar u_{\a^\e}^\e$ inside $B_\e(\a^\e)$.
By applying Lemma \ref{lemmaduenove} to $w^\e$ with $a$ replaced by $a^\e$, estimate \eqref{ce} reads
\begin{equation}\label{ce1}
\int_{B_\e(a^\e)}|\nabla w^\eps(x)|^2\,\de x\leq C\e^2\|g-\theta_{\a^\e}\|^2_{L^\infty(\partial\Om)},
\end{equation}
which implies that
\begin{equation}\label{cor_strong}
\|w^\eps\|_{H^{1}(B_\eps(\a^\e))}\rightarrow0,\qquad\text{as $\eps\rightarrow0$.}
\end{equation}
By combining \eqref{ce1} with \eqref{grad-n2}, 
we have
\begin{equation*}
\int_\Om |\nabla w^\e|^2\,\de x=  \int_{\Om_\e(\a^\e)}|\nabla \bar u^\eps_{\a^\e}|^2\,\de x+\int_{B_\eps(\a^\e)}|\nabla w^\eps|^2\,\de x
\leq  C\|g-\theta_{\a^\e}\|^2_{H^{1/2}(\partial\Om)}+C\e^2\|g-\theta_{\a^\e}\|^2_{L^\infty(\partial\Om)}.
\end{equation*}
Therefore, letting $\e\to0$, we obtain
\begin{equation*}
\limsup_{\e\to0} \int_\Om |\nabla w^\e|^2\,\de x \leq C\|g-\theta_{\a}\|^2_{H^{1/2}(\partial\Om)},
\end{equation*}
which, together with Poincar\'e inequality, implies that $w^\e$ is uniformly bounded in $H^1(\Om)$.
As a consequence  there exists $w\in H^1(\Om)$ such that (up to subsequences) $w^\e\rightharpoonup w$ weakly in $H^1(\Om)$.
Since $w^\e=g-\theta_{\a^\e}$ on $\partial\Om$ for every $\e$, then $w=g-\theta_\a$ on $\partial\Om$.
Since the $\bar u_{a^\e}^\e$ in \eqref{dec_eqn} is the minimizer of the Dirichlet energy,
we have that
\begin{equation*}
\frac12\int_{\Omega_\e(a^\e)}|\nabla \bar u^\e_{a^\e}|^2\,\de x \leq \frac12\int_{\Omega_\e(a^\e)}|\nabla v_{a^\e}|^2\,\de x \leq\frac12 \int_{\Om}|\nabla v_{a^\e}|^2\,\de x,
\end{equation*}
the lower semicontinuity of the $H^1$ norm, together with \eqref{cor_strong}, gives
\begin{equation}\label{lachiamo}
\frac12\int_\Om |\nabla w|^2\,\de x\leq \liminf_{\e\to0} \frac12\int_\Om |\nabla w^\e|^2\,\de x\leq \lim_{\e\to0}\frac12\int_\Om |\nabla v_{\a^\e}|^2\,\de x =\frac12\int_\Om |\nabla v_\a|^2\,\de x,
\end{equation}
which implies that $w=v_\a$, by the uniqueness of the minimizer $v_\a$. 
Thus, all the inequalities in \eqref{lachiamo} are in fact equalities.
This, together with \eqref{cor_strong}, gives \eqref{claim1} and completes the proof.
\end{proof}

\begin{proposition}\label{prop2}
Let $\a\in \partial\Om$ and $\a^\eps$ be a sequence of points in $\overline\Om$ converging to $\a$ as $\e\to0$. 
Then there exist two constants $C_1,C_2>0$ independent of $\e$, $\a^\e$, and $\a$, such that 
\begin{equation}\label{nuovatesi}
\mathcal F_\eps(\a^\e)\geq C_1|\log (\max\{\e,d(a^\e)\})|+C_2,
\end{equation}
for every $\epsilon$ small enough.
In particular, $\mathcal F_\eps(\a^\e)\rightarrow+\infty$ as $\e\to 0$.
\end{proposition}
\begin{proof}
Since $a\in\partial\Omega$, $\omega(a)=2$.
Let $\alpha$ and $\ea$ be as in \eqref{H3}, let $\e<\min\{\ea,1\}$, and $d^\e<\ea/2$, where we set, for brevity, $d:=d(\a)$ and $d^\eps:=d(\a^\e)$. 
We distinguish two possible scenarios: the \emph{slow collision} $\eps<d^\eps$ and the \emph{fast collision}  $\eps\geq d^\eps$. 
In the former case $\eps<d^\eps$, exploiting \eqref{formercase} we get
\begin{equation}\label{F1}
\begin{split}
\mathcal F_\eps(\a^\e) & \geq \pi\log d^\e+\frac{1}{2}\int_{\Om_{d^\e}(\a^\e)}|\k_{a^\e}|^2\,\de x+\int_{\Om_{d^\e}(\a^\e)} \k_{a^\e}\cdot\nabla\bar u_{\a^\e}^\e \,\de x + \frac{1}{2}\int_{\Om_{d^\e}(\a^\e)}|\nabla \bar u_{\a^\e}^\e|^2\,\de x \\
&=\pi\log d^\e+\frac{1}{2}\int_{\Om_{d^\e}(\a^\e)}|\k_{a^\e}+\nabla \bar u_{\a^\e}^\e|^2\,\de x.
\end{split}
\end{equation}
where we have used that $\int_{\Om_{\e}(\a^\e)} \k_{a^\e}\cdot\nabla\bar u_{\a^\e}^\e \,\de x=\int_{\Om_{d^\e}(\a^\e)} \k_{a^\e}\cdot\nabla\bar u_{\a^\e}^\e \,\de x$.
By Lemma~\ref{lem-data}, we may assume that {the discontinuity point $\b^\eps$ of the boundary datum $g$ is one of the projections of $\a^\eps$ on $\partial\Om$}, so that $\b^\eps\in\partial B_{d^\e}(\a^\e)\cap \partial \Omega$. 
In particular, $\Omega_{2d^\e}(\b^\e) \subset \Omega_{d^\e}(\a^\e)$, so that
 \begin{equation}\label{F3}
 \begin{split}
&\int_{\Om_{d^\e}(\a^\e)}  |\k_{\a^\e} + \nabla \bar u_{\a^\e}^\e|^2\,\de x\geq \int_{\Omega_{2d^\e}(\b^\e)}|\k_{\a^\e} + \nabla \bar u_{\a^\e}^\e|^2\,\de x \\
&\geq \inf\bigg\{\int_{\Omega_{2d^\e}(\b^\e)}|\nabla u|^2\,\de x : u\in H^1(\Omega_{2d^\e}(\b^\e)), u= g\ \hbox{on }\partial \Omega \setminus\overline{B}_{2d^\e}(\b^\e)\bigg\} \\
&= \inf\bigg\{\int_{\Omega_{2d^\e}(\b^\e)}| 2 \k_{\b^\e} + \nabla u|^2\,\de x :  u\in H^1(\Omega_{2d^\e}(\b^\e)),  u= g -  2\theta_{\b^\e}\ \hbox{on }\partial \Omega \setminus \overline{B}_{2d^\e}(\b^\e)\bigg\} \\ 
&=\int_{\Omega_{2d^\e}(\b^\e)}|  2 \k_{\b^\e} + \nabla \bar{u}|^2\,\de x,
\end{split}
\end{equation}
where $\bar{u}$ solves 
$$
\begin{cases}
\Delta \bar{u} = 0 \quad & \hbox{in }\Omega_{2d^\e}(\b^\e),
\\
\bar{u} = g- 2\theta_{\b^\e} \quad & \hbox{on }\partial \Omega \setminus \overline{B}_{2d^\e}(\b^\e),
\\
\partial \bar{u} / \partial \nu = -  2 \k_{\b^\e}\cdot \nu \quad & \hbox{on } \partial B_{2d^\e}(\b^\e)\cap \Omega.
\end{cases}
$$
Since $\k_{\b^\e}\cdot \nu =0$ on $ \partial B_{2d^\e}(\b^\e)$ by \eqref{euler-k}, it follows by uniqueness that $\bar{u}=\bar u_{b^\e}^\e$,  where $\bar u^\e_\b\in H^1(\Om_\e(\b))$ is the solution to \eqref{dec_eqn}.
Therefore, using Young's inequality and \eqref{grad-n2} (with $h=g-2\theta_b$), recalling that $\e< d^\epsilon$,  we may bound \eqref{F3} as follows:
\begin{equation}\label{F5}
\begin{split}
\frac{1}{2}\int_{\Omega_{2d^\e}(\b^\e)}  |2\k_{\b^\e} + & \nabla \bar{u}|^2\,\de x \geq  2 \int_{\Omega_{2d^\e}(\b^\e)}|\k_{\b^\e}|^2\,\de x + 2 \int_{\Omega_{2d^\e}(\b^\e)} \k_{\b^\e} \cdot \nabla \bar{u} \,\de x  \\
\geq & 2 \int_{\Omega_{2d^\e}(\b^\e)}|\k_{\b^\e}|^2\,\de x - \lambda \int_{\Omega_{2d^\e}(\b^\e)} |\k_{\b^\e}|^2 \,\de x- \frac{1}{\lambda} \int_{\Omega_{2d^\e}(\b^\e)}|\nabla \bar{u}|^2\,\de x  \\
\geq & (2-\lambda) \int_{\Omega_{2d^\e}(\b^\e)}|\k_{\b^\e}|^2\,\de x - \frac{C}\lambda\max_{b\in\partial\Om}\|g-2\theta_{\b}\|^2_{H^{1/2}(\partial\Om)}, 
\end{split}
\end{equation}
where $\lambda>0$ is a constant that will be chosen later.
Now, in view of the assumption on $d^\e$ at the beginning of the proof, the set $\Omega_{2d^\e}(\b^\e)$ contains a sector, which, in polar coordinates centered at $\b^\e$, is the rectangle
$R:= (2d^\e,\ea)\times(\phi_0, \phi_1)$ with $\phi_1-\phi_0=\alpha$.
Therefore, in the case $\e<d^\e$, by combining \eqref{F1} and \eqref{F3} with \eqref{F5}, and using the estimate from below in \eqref{solok1}, it follows that
\begin{equation*}
\mathcal F_\e(\a^\e)  \geq  \big(\pi-(2-\lambda)\alpha\big)\log d^\eps  + (2-\lambda)\alpha \log\frac\ea2  - \frac{C}\lambda\max_{b\in\partial\Om}\|g-2\theta_{\b}\|^2_{H^{1/2}(\partial\Om)}.
\end{equation*} 
Recalling that $\alpha>\pi/2$ by \eqref{H3}, we can choose $\lambda= (\alpha - \pi/2)/\alpha$, so that the inequality above can be written as
\begin{equation}\label{F7}
 \mathcal F_\e(\a^\e)  \geq C_1 |\log d^\eps| +C_2 ,
\end{equation}
with
\begin{equation}\label{F77}
C_1:= \Big(\alpha-\frac{\pi}{2}\Big) \quad \text{ and } \quad C_2:= \Big(\alpha+\frac{\pi}{2}\Big) \log\frac\ea2  - \frac{C\alpha}{\alpha-\pi/2}\max_{b\in\partial\Om}\|g-2\theta_{\b}\|^2_{H^{1/2}(\partial\Om)}.
\end{equation}

Moreover, in the case $\e\geq d^\e$, we consider the projection $\b^\e$  of $\a^\e$ on $\partial\Om$, so that $\b^\e\in \partial B_{\e}(a^\e)\cap\partial\Om$, and the disk $B_{2\e}(b_\e)$ contains $B_{\e}(a^\e)$. 
Arguing as in \eqref{F3}, \eqref{F5}, and using \eqref{solok1} with $d^\e$ replaced by $\e$, we can estimate \eqref{Feps} as follows
\begin{equation}
\mathcal F_\eps(\a^\e)\geq \frac{1}{2}\int_{\Om_{2\e}(\b^\e)}|2\k_{\b^\e} + \nabla \bar u|^2\,\de x-\pi|\log\e|\geq  C_1|\log \eps|  + C_2,\label{F11}
\end{equation}
with the same constants $C_1$ and $C_2$ provided in \eqref{F77}. 

Therefore, by \eqref{F7} and \eqref{F11} for every $\e<\min \{\ea,1\}$, the thesis \eqref{nuovatesi} follows.
\end{proof}

\subsection{\textbf{The case $n>1$ with $\min_i d_i>0$}}\label{ngen1}
In this case, it is convenient to write the rescaled energy \eqref{Fepsn} as the sum of the rescaled energy of each dislocation plus a remainder term accounting for interactions: recalling the expression \eqref{Feps} for the rescaled energy of one dislocation and the decomposition \eqref{dec1}, we write the energy \eqref{Fepsn} as 
\begin{equation}\label{nonmiscordare}
\mathcal F_\epsilon^{(n)}(\ba)=\sum_{i=1}^n \mathcal F_\epsilon(a_i) + \sum_{i=1}^n\mathcal R_\epsilon(a_i) + \sum_{i\neq j} \mathcal G_\epsilon(\a_i,\a_j),
\end{equation}
where, for every $i=1,\ldots,n$, $\mathcal{F}_\e(a_i)$ is given by \eqref{Feps},
\begin{equation}\label{reps}
\mathcal R_\epsilon(a_i)\coloneqq\frac12\int_{\Omega_\e(\ba)} |\k_{a_i}+\nabla \bar v_{a_i}^\epsilon|^2\, \de x- \frac12 \int_{\Omega_\e(\a_i)} |\k_{a_i}+\nabla \bar u_{a_i}^\epsilon |^2\, \de x,
\end{equation}
and, for every $i,j=1,\ldots,n$, with $i\neq j$,
\begin{equation*}
\mathcal G_\epsilon(a_i,a_j)\coloneqq\int_{\Omega_\e(\ba)}(\k_{a_i}+ \nabla \bar v_{a_i}^\epsilon)\cdot (\k_{a_j}+ \nabla \bar v_{a_j}^\epsilon)\, \de x,
\end{equation*}
$\bar u^\e_{a_i}$ being the solution to \eqref{dec_eqn} associated with $a_i$, and $\bar v_{a_i}^\e$ being as in Remark~\ref{remuq}.

\begin{proposition}\label{propR}
Let $\ba=(\a_1,\dots,\a_n)\in \Om^n$ be an $n$-tuple of distinct points and let $\ba^\e$ 
be a sequence converging in $\Om^n$ to $\ba$ 
as $\e\to0$. Then, for every $i=1,\dots, n$, we have 
$\mathcal R_\epsilon(a_i^\epsilon)\rightarrow 0$ as $\e\to0$.
\end{proposition}
\begin{proof}
For brevity, we define the $d_i^\e$'s as in \eqref{dconi}, associated with the family $\{a_1^\e,\ldots,a_n^\e\}$, and the $d_i$'s associated with the family $\{a_1,\ldots,a_n\}$.
By assumption $d_i^\e\to d_i>0$ for every $i=1,\ldots,n$; therefore, without loss of generality, we may take $\e< \min_i d_i$.
In particular, the disks $B_\e(a_i^\e)$ are all contained in $\Om$ and are pairwise disjoint for every $\e\in (0,\min_i{d_i})$. 

Fix now $i\in\{1,\ldots,n\}$. 
The remainder \eqref{reps} evaluated at $a^\e_i$ can be written as $\mathcal R_\epsilon(a_i^\e)=\mathcal R'_\e(a_i^\e)- \mathcal R''_\e(a_i^\e)$, with
$$
\mathcal R'_\e(a_i^\e)\coloneqq \frac12\int_{\Omega_\e(\ba^\e)}\big(2\k_{i}^\e+\nabla \bar v_{i}^\epsilon + \nabla \bar u_{i}^\epsilon\big)\cdot\big(\nabla \bar v_{i}^\epsilon- \nabla \bar u_{i}^\epsilon\big)\, \de x,
$$
and
$$
\mathcal R''_\e(a_i^\e)\coloneqq \frac{1}{2}\sum_{j\neq i}\int_{B_\e(a_j^\e)}|\k_{i}^\e+\nabla \bar u_{i}^\epsilon|^2\, \de x,
$$
where, for brevity, we have replaced the subscript $\a_j^\e$ with the subscript $j$ coupled with the superscript $\e$.
By the Divergence Theorem, \eqref{dec_eqn}, and Remark~\ref{remuq}, we have
\[
\begin{split}
\mathcal R'_\e(a_i^\e)
&=\frac12\int_{\partial\Om_\epsilon(\ba^\e)} \big(-\partial \bar v_{i}^\epsilon/\partial\nu + \partial \bar u_{i}^\epsilon/\partial\nu\big) (\bar v_i^\e-\bar u_i^\e)\, \de x\\
&=\frac12\int_{\Omega_\e(\ba^\e)}\big(-\nabla \bar v_{i}^\epsilon+\nabla \bar u_{i}^\epsilon\big)\cdot\big(\nabla \bar v_{i}^\epsilon- \nabla \bar u_{i}^\epsilon\big)\, \de x=
-\frac12\int_{\Omega_\e(\ba^\e)}\big|\nabla \bar v_{i}^\epsilon- \nabla \bar u_{i}^\epsilon\big|^2\, \de x.
\end{split}
\]
Then,
$\mathcal R'_\e(a_i^\e)\leq\|\nabla \bar v_{i}^\epsilon- \nabla v_{i}^\epsilon\big\|_{L^2(\Omega_\e(\ba^\e);\R^2)}+\|\nabla  v_{i}^\epsilon- \nabla \bar u_{i}^\epsilon\big\|_{L^2(\Omega_\e(\ba^\e);\R^2)}$,
which, in view of Lemma \ref{conservative-n}, 
converges to $0$, as $\e\to0$.  
Moreover, in view of Lemma~\ref{max_princ_n} (with $n=1$ and $h=g_{b_i^\e}-\theta_i^\e$) and the fact that $\|\k_i^\e\|_{L^\infty( B_\e(\a_j^\e);\R^2)}\leq 1/(2d_i^\e-\epsilon)$ for every $j\neq i$, we may bound $\mathcal R''_\e(a_i^\e)$ as follows:
$$
\mathcal R''_\e(a_i^\e) \leq \sum_{j\neq i} \int_{B_\e(a_j^\e)} \big(|\k_i^\e|^2 + |\nabla \bar u^\e_i|^2\big)\, \de x \leq
\pi(n-1) \e^2 \bigg(\frac{1}{(2d_i^\e-\e)^2} + C \|g_{b_i^\e}- \theta^\e_{i}\|^2_{L^\infty(\partial \Om)} \bigg),
$$
for some constant $C>0$ independent of $\e$. In particular $\mathcal R''_\e(a_i^\e)$ tends to zero as $\epsilon\to 0$. This concludes the proof of the proposition.
\end{proof}

\begin{proposition}\label{propG}
Let $\ba=(\a_1,\dots,\a_n)\in \Om^n$ be an $n$-tuple of distinct points and let $\ba^\e$ 
be a sequence converging in $\Om^n$ to $\ba$ 
as $\e\to0$. Then, for every $i,j = 1,\dots, n$, with $i\neq j$, we have 
\begin{equation}\label{convG}
\mathcal G_\epsilon(a_i^\epsilon,\a_j^\e)\rightarrow \int_{\Omega}(\k_{a_i}+ \nabla v_{a_i})\cdot (\k_{a_j}+ \nabla v_{a_j})\, \de x, \quad\hbox{as }\e\to0,
\end{equation}
$v_{a_i}$ being the solution to \eqref{def-barv-a} associated with $a_i$.
\end{proposition}
\begin{proof}
For brevity, we define the $d_i^\e$'s as in \eqref{dconi}, associated with the family $\{a_1^\e,\ldots,a_n^\e\}$, and the $d_i$'s associated with the family $\{a_1,\ldots,a_n\}$.
Fix $i,j\in\{a,\ldots,n\}$, $i\neq j$. By assumption $d_i^\e\to d_i>0$ and $d_j^\e\to d_j>0$; therefore, without loss of generality, we may take $\e< \min \{d_i,d_j\}$. Since the limit points 
are distinct, in view of Lemma~\ref{conservative-n}, we have 
$$
\chi_{\Omega_\e(\ba^\e)} \nabla\barv_{\a_i^\e}\to \nabla v_{a_i}\quad\text{and}\quad
\chi_{\Omega_\e(\ba^\e)} \nabla\barv_{\a_j^\e}\to \nabla v_{a_j}\quad \hbox{strongly in } L^2(\Om;\R^2),
$$
so that
\begin{equation}\label{convg1}
\int_{\Omega_\e(\ba^\e)}  \nabla\barv_{\a_i^\e}\cdot  \nabla\barv_{\a_j^\e}\,\de x \to \int_\Om \nabla v_{\a_i}\cdot  \nabla v_{\a_j}\,\de x .
\end{equation}
Setting for brevity $d:=\min \{d_i,d_j\}$, we decompose the domain of integration as
$$
\Omega_\e(\ba^\e) = \big(\Omega_{d}(a_i^\e,a_j^\e) \cup B_{d}(a_i^\e) \cup B_{d}(a_j^\e) \big) \setminus \bigcup_{k=1}^n \overline B_\e(a_k^\e).
$$
Since $\k_{\a_i^\e}\to \k_{a_i}$ and $ \k_{\a_j^\e}\to \k_{a_j}$ a.e.\ in $\Om$, by the Dominated Convergence Theorem it is easy to see that, as $\e\to 0$,
\begin{equation}\label{k1}
\int_{\Om_d(a_i^\e,a_j^\e)} \k_{\a_i^\e}\cdot \k_{\a_j^\e} \,\de x\to \int_{\Om_d(a_i,a_j)} \k_{\a_i}\cdot \k_{\a_j}\,\de x
\end{equation}
and
\begin{equation}\label{k2}
\int_{B_{d}(a_i^\e)} \k_{\a_i^\e}\cdot \k_{\a_j^\e}\,\de x = 
\int_{B_{d}(a_i)} \frac{ \hat{\theta}_{a_i} \cdot \hat\theta_{a_j^\e - (a_i^\e-a_i)}}{|x-a_i|\, |x-a_j^\e + (a_i^\e-a_i)|} \, \de x \to \int_{B_d(a_i)} \k_{\a_i}\cdot \k_{\a_j}\,\de x.
\end{equation}
An analogous result holds exchanging the roles of $i$ and $j$.
Eventually, since the limit points $a_1,\ldots,a_n$ are distinct, we have, for $k\neq i,j$,
\begin{equation}\label{k3}
\int_ {B_\e(a_k^\e) }|\k_{\a_i^\e}\cdot \k_{\a_j^\e}| \,\de x \leq \frac{\pi \e^2}{(2d_i^\e-\e)(2d_j^\e-\e)} \to 0,\quad \hbox{as }\e\to 0,
\end{equation}
and
\begin{equation}\label{k4}
\int_ {B_\e(a_i^\e) }|\k_{\a_i^\e}\cdot \k_{\a_j^\e}| \,\de x \leq \frac{2\pi\e}{2d_j^\e-\e} \to 0,\quad \hbox{as }\e\to 0
\end{equation}
(and again, the same holds swapping the roles of $i$ and $j$.)
Notice that \eqref{k3} and \eqref{k4} are refined versions of Lemma \ref{scalare}(ii).
By combining \eqref{k1}, \eqref{k2}, \eqref{k3}, and \eqref{k4} we get
 \begin{equation}\label{convg2}
\int_{\Omega_\e(\ba^\e)}\k_{\a_i^\e}\cdot \k_{\a_j^\e} \,\de x \to  \int_{\Omega}\k_{\a_i}\cdot \k_{\a_j} \,\de x,\quad \hbox{as }\e\to 0.
\end{equation}
In order to study the asymptotic behavior as $\e\to 0$ of the $L^2$ product of  $\k_{a_i^\e}$ and $\nabla\bar v ^\e_{a_j^\e}$ we use the decomposition $\bar v^\e_{a_j^\e} = \hat u_j^\e + q_j^\e$ introduced in Remark \ref{remuq}. We recall in particular Remark \ref{rem-conv}(i): $\hat u^\e$ admits an $H^1$ extension (not relabeled) that strongly converges to $v_{a_j}$; thus, integrating by parts and exploiting again the Dominated Convergence Theorem, in the limit as $\e \to 0$ we get
\begin{equation}\label{convg3}
\begin{split}
\int_{\Om_\e(\ba^\e)} \k_{a_i^\e}\cdot \nabla \hat u_j ^\e\,\de x &= \int_{\Om_\e(a_i^\e)}  \k_{a_i^\e}\cdot \nabla \hat u_j ^\e\,\de x - \int_{\bigcup_{j\neq i} B_\e(a_j^\e)} \k_{a_i^\e}\cdot \nabla \hat u_j ^\e\,\de x  
\\
&= \int_{\partial \Om} \k_{a_i^\e}\cdot \nu (g_{b_j^\e}-\theta_{a_j^\e})\,\de x + o(1)
\\ &  \to  \int_{\partial \Om} \k_{a_i}\cdot \nu (g_{b_j}-\theta_{a_j})\,\de x =   \int_{\Om} \k_{a_i}\cdot \nabla v _{a_j}\,\de x,
\end{split}
\end{equation}
where we have used the fact that $\hat u_j^\e$ has vanishing $L^2$ norm in the disks $B_\e(a_k^\e)$ as $\e \to 0$; while $\k_{a_i^\e}$  is uniformly bounded  in every disk  $B_\e(a_k^\e)$ with $k\neq i$, and satisfies $\k_{a_i^\e}\cdot \nu =0$ on $\partial B_\e(a_i^\e)$.
On the other hand, we recall that $(\nabla q_j^\e)^\perp = \nabla p_j^\e$ in the perforated domain, $p_j^\e$ being a solution to \eqref{5ottobre}. Since the solution to \eqref{5ottobre} is unique up to a constant, we choose $p_j^\e$ such that $ p_j^\e=0$ on $\partial B_\e(a_i^\e)$. Therefore, in view of Remark \ref{rem-conv}(ii), we infer that  $p_j^\e$ admits a an extension in $H^1(\Om)$ (not relabeled)  which is harmonic in every disk $B_\e(a_k^\e)$ and such that $\|p_j^\e\|_{H^1(\Om)}\rightarrow0$ as $\e\to0$.
Therefore, by letting $\phi_i^\e(x):=\log (|x- a_i^\e|)$, we have
{
\begin{equation}\label{convg4}
\begin{split}
& \int_{\Om_\e(\ba^\e)} \k_{a_i^\e}\cdot \nabla q_j^\e\,\de x = \int_{\Om_\e(\ba^\e)}  \nabla \phi_i^\e \cdot \nabla p_j^\e\,\de x = \int_{\bigcup_{k=1}^n \partial B_\e(a_k^\e)} \phi_i^\e \nabla p_j^\e\cdot \nu \, \de x\\
& =  \int_{\bigcup_{k=1}^n \partial B_\e(a_k^\e)} \phi_i^\e \nabla (p_j^\e - p_j^\e(a_i^\e))\cdot \nu \, \de x  = \int_{\bigcup_{k\neq i} B_\e{(a_k^\e)}} \nabla  \phi_i^\e\cdot \nabla p_j^\e\,\de x,
\end{split}
\end{equation}
and its absolute value can be estimated from above by 
\[
\frac{(n-1)\sqrt{\pi} \e}{d_i-\e}\|\nabla p_j^\e\|_{L^2(\Om)}\to 0.
\]
}
Similarly, the same limits in \eqref{convg3} and \eqref{convg4} hold exchanging the roles of $i$ and $j$. 
The thesis \eqref{convG} follows then by putting together \eqref{convg1}, \eqref{convg2}, \eqref{convg3}, and \eqref{convg4}.
\end{proof}

\subsection{\textbf{The case $n>1$ with $\min_i d_i=0$}}\label{ngen2}

\begin{lemma}\label{lemmafondamentale1}
Let $\ba=(a_1,\ldots,a_n)\in\overline\Omega{}^n$ and let $0<\e<\eta$ be such that for every $a_j\in\Om$ we have $B_\eta(a_j)\subset\Om$  (i.e. $d_j\geq\eta$).
Then there exists a positive constant $C$, independent of $\e$ and $\eta$, such that
\begin{equation}\label{fondamentale1}
\mathcal F_\e^{(n)}(\ba)\geq \mathcal F_{\eta}^{(n)}(\ba)- C.
\end{equation}
\end{lemma}
\begin{proof}
We start by comparing the energies $\mathcal E_\e$ and $\mathcal E_\eta$.
Recalling \eqref{ginevra}, it is easy to see that
\begin{equation}\label{EJ}
 \mathcal E_\eps^{(n)}(\ba) \geq \mathcal E_{\eta}^{(n)} (\ba) +\frac12  \int_{\cup_{i=1}^\ell A_\e^\eta(a_i)} \Big|\sum_{j=1}^n \omega(a_j) \k_{\a_j}+\nabla \barv_{\ba}\Big|^2\de x 
 =  \mathcal E_{\eta}^{(n)} (\ba) + J_1+J_2+J_3+J_4,
\end{equation}
where we have defined
\begin{equation*}
\begin{split}
J_1:= \frac12  \sum_{i=1}^\ell \int_{A_\e^\eta(a_i)} |\nabla \barv_{\ba}|^2 \, \de x, &\quad
J_3:= \sum_{i,j,h=1 \atop h>j}^\ell  m_j m_h \omega(a_j)\omega(a_h) \int_{A_\e^\eta(a_i)} \k_{\a_j}\cdot \k_{a_h}\, \de x, \\
J_2:= \frac12  \sum_{i,j=1}^\ell m_j^2 \omega^2(a_j)\int_{A_\e^\eta(a_i)} |\k_{\a_j}|^2 \, \de x, &\quad
J_4:= \sum_{i,j=1}^\ell m_j \omega(a_j) \int_{A_\e^\eta(a_i)}  \k_{\a_j}\cdot \nabla \barv_{\ba}\de x.
\end{split}
 \end{equation*}

The term $J_1$ is strictly positive, and we can neglect it.
In order to bound $J_2$ from below, we need to distinguish two cases, according to whether $d(a_i)$ is greater or smaller than $\eta$, namely whether or not the annulus $A_\e^\eta(a_i)$ is contained in $\Om$.
In the former case, $\omega(a_i)=1$ and a simple computation gives (see \eqref{solok})
\begin{equation}\label{maggiore}
\frac12  \sum_{j=1}^\ell m_j^2 \omega^2(a_j) \int_{A_\e^\eta(a_i)} |\k_{\a_j}|^2 \, \de x \geq \frac12 m_i^2 \omega^2(a_i)  \int_{A_\e^\eta(a_i)} |\k_{\a_i}|^2 \, \de x = \pi m_i^2 \log (\eta/\e).
\end{equation}
In the latter, $A_\e^\eta(a_i)$ is not completely contained in $\Om$ and therefore, by hypothesis, we know that $a_i\in\partial\Om$ and in particular $\omega(a_i)=2$, so that 
\begin{equation}\label{uguale}
\frac12  \sum_{j=1}^\ell m_j^2 \omega^2(a_j) \int_{A_\e^\eta(a_i)} |\k_{a_j}|^2 \, \de x \geq 2m_i^2  \int_{A_\e^\eta(a_i)} |\k_{a_i}|^2 \, \de x > \pi m_i^2 \log \Big(\frac{\min(\eta,\ea)}{\e}\Big),
\end{equation}
where $\ea$ is given in \eqref{H3} (see \eqref{solok1}). 
Therefore, for $\eta$ small enough, $\min(\eta,\ea)=\eta$ and \eqref{uguale} provides the same bound as \eqref{maggiore}, hence, summing over $i$, we obtain
\begin{equation}\label{boundJ_2}
J_2\geq\sum_{i=1}^\ell \pi m_i^2\log(\eta/\e).
\end{equation}
Recalling the definition \eqref{Fepsn} of $\mathcal F_\e^{(n)}(\ba)$, from \eqref{EJ} 
and \eqref{boundJ_2} we obtain
\begin{equation}\label{intermediate}
\mathcal F_\e^{(n)}(\ba)\geq  \mathcal F_\eta^{(n)}(\ba)+ \Big(\sum_{i=1}^\ell m_i^2-n\Big)\pi\log(\eta/\e)+J_3+J_4 \geq  \mathcal F_\eta^{(n)}(\ba)+J_3+J_4,
\end{equation}
since $\eta>\e$ and $\sum_{i=1}^\ell m_i^2-n\geq0$.
We obtain the thesis \eqref{fondamentale1} from \eqref{intermediate}, provided we bound $J_3$ and $J_4$ from below.

In view of Lemma \ref{scalare}(i), we have
\begin{equation}\label{mag1}
J_3= \sum_{i,j,h=1 \atop h>j}^\ell  m_j m_h \omega(a_j)\omega(a_h) \int_{A_\e^\eta(a_i)} \k_{\a_j}\cdot \k_{a_h}\, \de x  \geq - 8\pi\ell n^2 \geq -8\pi n^3. 
\end{equation}

To estimate $J_4$, we start by splitting $\nabla \barv_{\ba}=\nabla \hat u_{\ba}^\e+\nabla q^\e_{\ba}$, 
with $\hat u_{\ba}^\e=\sum_{k=1}^n \hat u_{a_k}^\e$ and $q_{\ba}^\e=\sum_{k=1}^n q_{a_k}^\e$, where $\hat u_{a_k}^\e$ and $q_{a_k}^\e$ are introduced in Remark~\ref{remuq}.
Then
$$
\sum_{i=1}^\ell\int_{A_\e^\eta(a_i)}  \k_{\a_j}\cdot \nabla \barv_{\ba}\de x =  \sum_{i=1}^\ell \int_{A_\e^\eta(a_i)}  \k_{\a_j}\cdot \nabla \hat u_{\ba}^\e\de x +\sum_{i=1}^\ell \int_{A_\e^\eta(a_i)} \k_{\a_j}\cdot\nabla q_{\ba}^\e \de x
$$
and, by using the Divergence Theorem, \eqref{grad-n}, and the fact that $|\k_{\a_j}(x)|\leq |x|^{-1}$ on $\partial A_\e^\eta(a_i)$, we can estimate
\begin{equation}\label{mag22}
\begin{split}
\sum_{i=1}^\ell \int_{A_\e^\eta(a_i)}  \k_{\a_j}\cdot \nabla \hat u_{\ba}^\e\,\de x= & \sum_{i=1}^\ell \int_{\partial A_\e^\eta(a_i)}  \k_{\a_j}\cdot \nu \, \hat u_{\ba}^\e\,\de x \\
\geq & -\bar C\sum_{i=1}^\ell\int_{\partial A_\e^\eta(a_i)}  |\k_{\a_j}\cdot \nu |\,\de x\geq-2\pi\bar Cn,
\end{split}
\end{equation}
where $\bar C:=\left\|g- \sum_{k=1}^n\omega(a_k)\theta_{a_k}\right\|_{L^\infty(\partial\Om)}$.
Moreover 
\begin{align*}
 \sum_{i=1}^\ell \int_{A_\e^\eta(a_i)}  \k_{\a_j}\cdot\nabla q_{\ba}^\e \,\de x=\sum_{i=1}^\ell \int_{A_\e^\eta(a_i)}  \nabla \phi_{\a_j}\cdot\nabla p_{\ba}^\e\, \de x, 
\end{align*}
with $p_{\ba}^\eps=\sum_{i=1}^n p_{a_i}^\e$, where $p_{a_i}^\e$ solves \eqref{5ottobre}. 
In particular, by using \eqref{latesi1}, we infer that the $L^\infty$ norm of $p_{\ba}^\eps$ is bounded by $Cn\e$, so that we can estimate
\begin{align}\label{mag3}
\sum_{i=1}^\ell \int_{A_\e^\eta(a_i)}  \k_{\a_j}\cdot\nabla q_{\ba}^\e \de x\geq -2\pi Cn^2.
\end{align}

Combining \eqref{mag22} with \eqref{mag3} and summing over $j$, we obtain
$J_4\geq-2\pi n^2(\bar C+Cn)$, 
which, together with \eqref{intermediate} and \eqref{mag1}, allows us to get estimate \eqref{fondamentale1}, with constant $C=2\pi n^2\big(\bar C+(C+4)n\big)$.
The lemma is proved.
\end{proof}

\begin{proposition}\label{ultima-prop}
Let $\ba\in \overline \Om{}^n$ 
and let $\ba^\e$ 
be a sequence of points in $\overline{\Om}{}^n$ converging to $\ba$ 
as $\e\to0$. If $\min_{1\leq i\leq n} d_i=0$ with $d_i$ defined as in \eqref{dconi} then $\mathcal F_\epsilon^{(n)}(\ba^\e)\to \infty$, as $\eps\to0$.
\end{proposition}
\begin{proof}
{Our goal is to show that we can bound the energy $\mathcal F_\epsilon^{(n)}(\ba^\e)$ from below by a quantity that explodes in the limit as $\e\to0$.
This will be achieved by applying the following iterative procedure, which is performed at $\e$ fixed.}

\medskip

\noindent\textbf{Step 0 (Labeling)}  
We start by relabeling in a more suitable way the limit dislocations and the approximating ones. According to Definition~\ref{def-mi}, we relabel the limit points so that the first $\ell$ ($1\leq \ell\leq n$) are distinct. Moreover, we fix $\delta>0$ such that the disks 
$B_{\delta} (a_i)$ are pairwise disjoint for $i=1,\ldots,\ell$.
For every $i=1,\ldots,\ell$, there are $m_i$ points in $\{a_1^\e,\ldots,a_n^\e\}$ which converge to $a_i$. We denote these points by $a^\e_{i,j}$, with $j=1,\ldots,m_i$. For $\e$ small enough we clearly have $a^\e_{i,j}\in B_\delta(a_i)$, for every $j=1,\ldots,m_i$ and $i=1,\ldots,\ell$. 

\medskip

{At each iteration (from Step 1 to Step 3), we will replace the sequence (with respect to $\e>0$) of families (indexed by $i\in\{1,\ldots,\ell\}$) $\{a^\e_{i,j}\}_{j=1}^{m_i}$ with a sequence of singletons $c_i^\e$, still converging to $a_i$ as $\e \to 0$, with multiplicity $m_i$.
Additionally, we will define a new core radius $\eta(\e)$ and in Step 4 we will compare the energies $\mathcal F_\epsilon^{(n)}(\ba^\e)$ and $\mathcal F_{\eta(\epsilon)}^{(n)}(\bc^\e)$.}

\medskip

\noindent\textbf{Step 1 (Ordering)} Let $i\in \{1,\ldots,\ell\}$ be fixed. 
According to Definition \ref{def-mi}, we order the family $\{a^\e_{i,j}\}_{j=1}^{m_i}$ so that the first $\ell^\e_i$ are distinct, and we denote by $m_{i,j}^\e$ their multiplicities.
Notice that $\sum_{j=1}^{\ell_i^\e} m_{i,j}^\e=m_i$.
With these positions, we clearly have
\begin{equation}\label{easy}
\mathcal F_\e^{(n)}(\ba^\e) = \mathcal F_\e^{(n)} (\underbrace{a_{1,1}^\e}_{m_{1,1}^\e\text{-times}},\ldots,\underbrace{a_{1,\ell_1^\e}^\e}_{m_{1,\ell_1^\e}^\e\text{-times}},\ldots,\underbrace{a_{i,1}^\e}_{m_{i,1}^\e\text{-times}},\ldots,\underbrace{a_{\ell,1}^\e}_{m_{\ell,1}^\e\text{-times}},\ldots,\underbrace{a_{\ell,\ell_\ell^\e}^\e}_{m_{\ell,\ell_\ell^\e}^\e\text{-times}}).
\end{equation}

In case $\ell^\e_i>1$, for every $j=1,\ldots,\ell^\e_i$, we associate to the distinct points $a^\e_{i,j}$ the following quantity:
if $a^\e_{i,j}\in \partial \Om$, we set
\begin{equation*}
s(a^\e_{i,j}):=   \min\bigg\{\frac{|a^\e_{i,j}-a^\e_{i,k}|}{2}: k\in\{1,\ldots,\ell^\e_i\}\setminus\{j\}\bigg\},
\end{equation*}
whereas, if $a^\e_{i,j}\in\Om$, we set
\begin{equation*}
s(a^\e_{i,j}):= \min\bigg\{d(a^\e_{i,j}),  \min\bigg\{\frac{|a^\e_{i,j}-a^\e_{i,k}|}{2}: k\in\{1,\ldots,\ell_i^\e\}\setminus\{j\}\bigg\}\bigg\}.
\end{equation*}
Observe that, if the limit point $a_i\in\Om$,  for $\e$ small enough $d(a^\e_{i,j})$ is always greater than any mutual semidistance ${|a^\e_{i,j}-a^\e_{i,k}|}/{2}$, for all $j,k\in\{1,\ldots,\ell_i^\e\}$.
Up to  reordering the different $\ell^\e_i$ points $\{a^\e_{i,1},\ldots,a^\e_{i,\ell^\e_i}\}$ we can always suppose that
\begin{equation*}
0< s(a^\e_{i,1})\leq \ldots \leq s(a^\e_{i,\ell_i}). 
\end{equation*}

In case $\ell_i^\e=1$, namely when all the $a^\e_{i,j}$ coincide with $a^\e_{i,1}$, we set 
\begin{equation}\label{soluzionedituttiiproblemi}
s(a^\e_{i,1}):=
\begin{cases}
0 & \text{if $a_i\in\Omega$}, \\
d(a_{i,1}^\e) & \text{if $a_i\in\partial\Omega$}.
\end{cases}
\end{equation}

\medskip

\noindent \textbf{Step 2 (Stop test)}
If $s(a^\e_{i,1})=0$ for every $i=1,\ldots,\ell$, then we define $c_i^\e:=a_{i,1}^\e$, $\eta(\e):=\e$, and we go to Step 4.
Observe that by \eqref{soluzionedituttiiproblemi} $c_i^\e\in\partial\Om$ if $a_i\in\partial\Om$.
Otherwise, we define the following quantity
\begin{equation}\label{min-s}
\hat{s}=\hat s(\e):=\min\big\{s(a^\e_{i,1})>0 \ :\ i\in\{1,\ldots,\ell\}\big\}
\end{equation}
and we go to Step 3.
Observe that the set where the minimum is taken is not empty, hence $\hat s$ is finite and strictly positive.

\medskip 

\noindent \textbf{Step 3 (Iterative step)} We compare $\hat s$ with $\e$.

If $\hat{s}>\e$, we relabel $a_{k,j}^\e$ by $\hat a_{k,j}^\e$ and their multiplicities accordingly, set $\hat\e:=\hat s$, and estimate $\mathcal F_\e^{(n)}(\ba^\e)$ by means of \eqref{fondamentale1} proved in Lemma \ref{lemmafondamentale1}, with $\eta=\hat\e$, to obtain
\begin{equation}\label{fondamentale11}
\mathcal F_\e^{(n)}(\ba^\e)\geq \mathcal F_{\hat\e}^{(n)}(\ldots,\underbrace{\hat a_{k,j}^\e}_{\hat m_{k,j}^\e\text{-times}},\ldots)-C.
\end{equation}

If $\hat{s}\leq \e$ we distinguish two cases: 
\begin{itemize}
\item[(1)]  $\hat s$ is equal to $|a^\e_{i,1}-a^\e_{i,2}|/2$, for some $i$:
we replace the points $a^\e_{i,1}$ and $a^\e_{i,2}$ (and all those coinciding with either one of them) by $\hat a^\e_{i,1}$, the midpoint between $a^\e_{i,1}$ and $a^\e_{i,2}$, with multiplicity $\hat m_{i,1}^\e:=m_{i,1}^\e+m_{i,2}^\e$.
The replacement is performed simultaneously for all the indices $i$ that realize the minimum in \eqref{min-s}.
For all the other indices, we simply relabel $a_{k,j}^\e$ by $\hat a_{k,j}^\e$ and their multiplicities accordingly.
\item[(2)] $\hat s$ is realized by $d(a^\eps_{i,1})$, for some $i$:
we replace the point $a^\eps_{i,1}$ (and all those coinciding with it) by $\hat a^\e_{i,1}$, its projection to the boundary $\partial \Om$, with multiplicity $\hat m_{i,1}^\e:=m_{i,1}^\e$. 
\end{itemize}
Setting $\hat\e:=\e+\hat s$ and recalling \eqref{ginevra}, \eqref{Fepsn}, and \eqref{easy}, we have
\begin{equation}\label{fondamentale2}
\mathcal F_\e^{(n)}(\ba^\e)\geq \mathcal F_{\hat\e}^{(n)} (\ldots,\underbrace{\hat a_{k,j}^\e}_{\hat m_{k,j}^\e\text{-times}},\ldots)
-n\pi \log 2.
\end{equation}

Notice that $\hat\e$ satisfies the following bound:
\begin{equation}\label{fondamentale_e} 
\hat\e\leq \max\{2\e,\hat s\}\leq \max\{2\e,\bar s\},
\end{equation}
where $\bar s$ is defined as the maximum value of the $s(a_{i,j}^\e)$, namely
\begin{equation*}
\bar{s}=\bar s(\e):=\max\big\{s(a^\e_{i,\ell^\e_i})>0 \ :\ i\in\{1,\ldots,\ell\}\big\}.
\end{equation*}
We have obtained a new family $\{\hat a_{i,j}^\e\}$ and a new radius $\hat\e$ and we restart the procedure by applying Step 1 to these new objects.

\medskip

Notice that the procedure ends after at most $n^2$ iterations.
Indeed, when applying Step 3, we will always fall into case (1) after at most $n$ iterations, and then the number of distinct points will decrease when we apply (1). In conclusion, since the number of distinct points is at most $n$, we will reach the target situation after at most $n^2$ iterations of Step 3. 

\medskip

\noindent \textbf{Step 4 (Estimates and conclusion)} 
By combining the chain of inequalities obtained by applying Step 3 $k(\leq n^2)$ times, estimates \eqref{fondamentale2} and \eqref{fondamentale11} give 
\begin{equation}\label{esplode}
\mathcal F_\e^{(n)}(\ba^\e)\geq \mathcal F_{\eta(\e)}^{(n)}(\underbrace{c_1^\e}_{m_1\text{-times}},\ldots,\underbrace{c_\ell^\e}_{m_\ell\text{-times}})- k(C+n\pi\log2),
\end{equation}
where $\eta(\e)$ is a number depending on $\e$ obtained after $k$ iterations of the procedure that defines $\hat\e$ in Step 3.
Let us estimate $\eta(\e)$. 
At every iteration, the value $\bar s$ can increase, but it is easy to see that it cannot grow lager than its double.
Therefore, after $k$ iterations of Step 3, by \eqref{fondamentale_e}, we have
\begin{equation*}
\eta(\e)\leq 2^k\max\{\e,\bar s\}\leq 2^{n^2}\max\{\e,\bar s\}.
\end{equation*}
Since $\bar s$ tends to $0$ as $\e\to0$, we have $\eta(\e)\to 0$ as $\e\to0$.

We now claim that the right-hand side of \eqref{esplode} tends to $+\infty$ as $\e\to0$. 
Similarly to \eqref{nonmiscordare}, we write the energy as the sum of three contributions:
\begin{equation}\label{esplode2}
 \mathcal F_{\eta(\e)}^{(n)}(\underbrace{c_1^\e}_{m_1\text{-times}},\ldots,\underbrace{c_\ell^\e}_{m_\ell\text{-times}})=\sum_{i=1}^\ell\mathcal F_{\eta(\e)}^{(m_i)}(\underbrace{c_i^\e}_{m_i\text{-times}})+\sum_{i=1}^\ell m_i^2\mathcal R_\e(c_i^\e)+\sum_{i\neq j}m_im_j\mathcal G_\e(c_i^\e,c_j^\e)
\end{equation}
where, setting $\omega_i:=\omega(a_i)=\omega(c_i^\epsilon)$ (see Step 3), and $\k_i^\e:=\k_{c_i^\e}$ for brevity, 
\begin{equation}\label{rr}
\mathcal R_\e(c_i^\e):=\frac12\int_{\Om_{\eta(\e)}(c_1^\e,\dots,c_\ell^\e)}|\omega_i\k_i^\e+\nabla \bar v^\e_i|^2\,\de x-\frac12\int_{\Om_{\eta(\e)}(c_i^\e)}|\omega_i \k_i^\e+\nabla\bar u^\e_i|^2\,\de x,
\end{equation}
and
\begin{equation}\label{ii}
\mathcal G_\e(c_i^\e,c_j^\e):=\int_{\Om_{\eta(\e)}(c_1^\e,\dots,c_\ell^\e)}(\omega_i\k_i^\e+\nabla \bar v^\e_i)\cdot(\omega_j\k_j^\e+\nabla \bar v^\e_j)\,\de x,
\end{equation}
$\bar u^\e_i$ being the solution to \eqref{dec_eqn} 
associated with $c_i^\e$ and $\bar v_i^\e$ being as in Remark~\ref{remuq} for $i=1,\dots, \ell$ and core radius $\eta(\eps)$.

Fix $i\in\{1,\dots, \ell\}$. 
If $a_i\in \Om$, we have 
\begin{equation}\label{eqqdue}
\mathcal F_{\eta(\eps)}^{(m_i)}(\underbrace{c_i^\e}_{m_i\text{-times}})=m_i^2 \mathcal F_{\eta(\epsilon)}(c_i^\e)+m_i(m_i-1)\pi|\log \eta(\eps)|\geq m_i(m_i-1)\pi|\log \eta(\eps)|-C,
\end{equation}
where the equality follows from \eqref{ecco} and the inequality is a consequence of Proposition \ref{prop1}.
If instead $a_i\in\partial \Om$, by using \eqref{ecco} again and \eqref{nuovatesi}, we have
\begin{equation*}
\mathcal F_{\eta(\eps)}^{(m_i)}(\underbrace{c_i^\e}_{m_i\text{-times}})\geq C_1m_i^2|\log(\max\{\eta(\e),d(c_i^\e)\})|+m_i(m_i-1)\pi|\log\eta(\e)|+C_2m_i^2
\end{equation*}
(the constants $C_1$ and $C_2$ are those in \eqref{F77}), and, since $d(c_i^\e)=0$ as noticed in Step 2, we obtain
\begin{equation}\label{eqqquattro}
\mathcal F_{\eta(\eps)}^{(m_i)}(\underbrace{c_i^\e}_{m_i\text{-times}})\geq(C_1m_i^2+m_i(m_i-1)\pi)|\log\eta(\e)|+C_2m_i^2.
\end{equation}

By Lemma \ref{ext_lemma} the function $\bar v_i^\e$ can be extended inside any disk $B_{\eta(\e)}(c^\e_j)$ with $j\neq i$; therefore the remainder \eqref{rr} can be rewritten as $\mathcal R_\epsilon(c_i^\e)=\mathcal R'_\epsilon(c_i^\e)-\mathcal R''_\epsilon(c_i^\e)$ with
$$\mathcal R'_\epsilon(c_i^\e):=\frac12\int_{\Om_{\eta(\e)}(c_i^\e)}|\omega_i \k_i^\e+\nabla \bar v_i^\e|^2\de x-\frac12\int_{\Om_{\eta(\e)}(c_i^\e)}|\omega_i \k_i^\e+\nabla \bar u_i^\e|^2\de x
$$
and
$$\mathcal R''_\epsilon(c_i^\e):=\frac12\sum_{j\neq i}\int_{B_{\eta(\e)}(c_j^\e)}|\omega_i \k_i^\e+\nabla \bar v_i^\e|^2\de x.$$
Using the minimality of $\bar u_i^\e$ in $\Omega_{\eta(\e)}(c^\e_i)$ it turns out that $\mathcal R'_\epsilon(c_i^\e)\geq0$. 
Using that $|K_i^\e|\leq 1/\eta(\epsilon)$ in $B_{\eta(\e)}(c_j^\e)$ for every $j\neq i$ and Lemma \ref{ext_lemma} we obtain
\begin{equation}\label{stimaR}
\mathcal R_\epsilon(c_i^\epsilon)\geq - \mathcal R''_\epsilon(c_i^\e)\geq -\frac12\sum_{j\neq i}\int_{B_{\eta(\e)}(c_j^\e)}|\omega_i\k_i^\e|^2\de x-\frac12\sum_{j\neq i}\int_{B_{\eta(\e)}(c_j^\e)}|\nabla \bar v^\e_i|^2\de x\geq -C. 
\end{equation}

To estimate the interaction term \eqref{ii}, fix also $j\in\{1,\dots, \ell\}\setminus\{i\}$. 
Then, we can write $\mathcal G_\epsilon(c_i^\e,c_j^\e)=\mathcal G'_\epsilon(c_i^\e,c_j^\e)+\mathcal G''_\epsilon(c_i^\e,c_j^\e)+\mathcal G'''_\epsilon(c_i^\e,c_j^\e)$ with
\begin{subequations}
\begin{align}
\mathcal G'_\epsilon(c_i^\e,c_j^\e) &:=  \omega_i\omega_j\int_{\Om_{\eta(\e)}(c_1^\e,\dots,c_\ell^\e)}\k^\e_i\cdot \k^\e_j\,\de x, \label{inter1}\\
\mathcal G''_\epsilon(c_i^\e,c_j^\e)& :=  \omega_j\int_{\Om_{\eta(\e)}(c_1^\e,\dots,c_\ell^\e)}\nabla \bar v^\e_i\cdot \k_j^\e\,\de x+\omega_i\int_{\Om_{\eta(\e)}(c_1^\e,\dots,c_\ell^\e)}\nabla \bar v^\e_j\cdot \k_i^\e\,\de x, \label{inter2}\\
\mathcal G'''_\epsilon(c_i^\e,c_j^\e)& :=  \int_{\Om_{\eta(\e)}(c_1^\e,\dots,c_\ell^\e)}\nabla \bar v^\e_i\cdot\nabla \bar v^\e_j\,\de x. \label{inter3}
\end{align}
\end{subequations}
Thanks to Lemma \ref{scalare}(i), the functional $\mathcal G_\e'$ in \eqref{inter1} is uniformly bounded from below by a constant.
To estimate $\mathcal G''_\epsilon$ and $\mathcal G'''_\epsilon$ we recall the decomposition $\bar v_k^\e=\hat u_k^\e+q_k^\e$ and the function $p_k^\e$, solution to \eqref{5ottobre}, that we introduced in Remark \ref{remuq}, where $k=i$ or $k=j$.
Here the functions $\hat u_k^\e$ and $q_k^\e$ are introduced in Remark~\ref{remuq} with $(a_1^\e,\dots,a_n^\e)$ replaced by $(c_1^\e,\dots,c_\ell^\e)$, and the Dirichlet boundary condition for $\hat u_k^\e$ given by $g_{b_k^\e}-\omega_k\theta_k^\e$ on $\partial\Om$. Therefore, we can estimate \eqref{inter3} as follows
\begin{equation*}
\mathcal G'''_\epsilon(c_i^\e,c_j^\e)\leq \int_{\Om_{\eta(\e)}(c_1^\e,\dots,c_\ell^\e)}\big(|\nabla \hat u^\e_i|^2+|\nabla \hat u^\e_j|^2+|\nabla q^\e_i|^2+|\nabla q^\e_j|^2\big)\,\de x
\end{equation*}
and since, the gradient of $q_k^\e$ coincides in modulus with the gradient of $p_k^\e$, by \eqref{grad-n2} and \eqref{latesi} we can bound $\mathcal G'''_\epsilon(c_i^\e,c_j^\e)$ from below. 
The functional in \eqref{inter2} is the most delicate to treat: if both $a_i$ and $a_j$ are in $\Omega$, $\mathcal G_\e''(c_i^\e,c_j^\e)$ is bounded below by a constant (notice that in this case we could have applied Proposition \ref{propG} to $\mathcal G_\e(c_i^\e,c_j^\e)$ itself and we would have concluded).  In the general case, by means of the above-mentioned decomposition, \eqref{inter2} reads
\begin{equation}\label{inter21}
\begin{split}
\mathcal G''_\epsilon(c_i^\e,c_j^\e)&=\omega_j\int_{\Om_{\eta(\e)}(c_1^\e,\dots,c_\ell^\e)}\nabla \hat u^\e_i\cdot \k_j^\e\,\de x+\omega_i\int_{\Om_{\eta(\e)}(c_1^\e,\dots,c_\ell^\e)}\nabla \hat u^\e_j\cdot \k_i^\e\,\de x\\
&+\omega_j\int_{\Om_{\eta(\e)}(c_1^\e,\dots,c_\ell^\e)}\nabla  q^\e_i\cdot \k_j^\e\,\de x+\omega_i\int_{\Om_{\eta(\e)}(c_1^\e,\dots,c_\ell^\e)}\nabla  q^\e_j\cdot \k_i^\e\,\de x.
\end{split}
\end{equation}
We first deal with the last two terms, involving the gradients of $q^\epsilon_i$ and $q^\epsilon_j$: from H\"older's inequality, recalling that the gradient of $q^\e_k$ coincides in modulus with the gradient of $p^\e_k$, by  \eqref{latesi} and \eqref{solok} we conclude that these terms are bounded. To estimate the first two terms in \eqref{inter21}, we use Young's inequality, \eqref{grad-n2} and \eqref{solok} to obtain
\begin{equation*}
\begin{split}
&\omega_j\int_{\Om_{\eta(\e)}(c_1^\e,\dots,c_\ell^\e)}\nabla \hat u^\e_i\cdot \k_j^\e\,\de x+\omega_i\int_{\Om_{\eta(\e)}(c_1^\e,\dots,c_\ell^\e)}\nabla \hat u^\e_j\cdot \k_i^\e\,\de x \\
&\leq\frac{1}{2\lambda}\int_{\Om_{\eta(\e)}(c_1^\e,\dots,c_\ell^\e)} (\omega_j|\nabla \hat u^\e_i|^2+\omega_i|\nabla \hat u^\e_j|^2)\,\de x+\frac{\lambda}{2}\int_{\Om_{\eta(\e)}(c_1^\e,\dots,c_\ell^\e)} (\omega_i|\k_i^\e|^2+\omega_j|\k_j^\e|^2)\,\de x\\
&\leq  \frac{C}{\lambda}\Big(\|g_{b^\e_i}-\omega_i\theta^\e_i\|^2_{H^{1/2}(\partial\Omega)}+\|g_{b^\e_j}-\omega_j\theta^\e_j\|^2_{H^{1/2}(\partial\Omega)}\Big)+2\pi\lambda|\log(\eta(\e))|+2\pi\lambda\log(\mathrm{diam}\,\Omega),\\
\end{split}
\end{equation*}
where $C$ is a positive constant independent of $\epsilon$ and $\lambda>0$ is an arbitrary constant that will be chosen later.
Finally, we can control the term $\mathcal G_\e(c_i^\e,c_j^\e)$ as
\begin{equation}\label{stimaG}
\mathcal G_\epsilon(c_i^\e,c_j^\e)\geq - 2\pi\lambda|\log(\eta(\e))|-C_{ij}
\end{equation}
where all the terms independent of $\e$ have been included in the constant $C_{ij}$. 

We can classify each point $a_i$ according to whether it belongs to
\begin{itemize}
\item[1.]  the interior of $\Om$, with multiplicity $m_i=1$;
\item[2.]  the interior of $\Om$, with multiplicity $m_i>1$;
\item[3.] the boundary $\partial \Om$.
\end{itemize}
For $k=1, 2, 3$ we denote by $I_k$ the set of indices corresponding to those points $c_i^\e\to a_i$ belonging to the $k$-th category.
Therefore, combining \eqref{eqqdue}, \eqref{eqqquattro}, \eqref{stimaR}, and \eqref{stimaG}, summing over $i$ and $j$, we obtain the following lower bound for the energy \eqref{esplode2}:
$$
\mathcal F_{\eta(\e)}^{(n)}(\underbrace{c_1^\e}_{m_1\text{-times}},\ldots,\underbrace{c_\ell^\e}_{m_\ell\text{-times}})\geq 
\pi(C_{\mathcal{F}}-\lambda C_{\mathcal{G}})|\log\eta(\eps)|+C,
$$
where
\[
C_{\mathcal F}:=\sum_{i\in I_2\cup I_3} m_{i}(m_{i}-1)+\sum_{i\in I_3} C_{1}m_{i}^2, \quad\quad C_{\mathcal G}:=2n^2,
\]
and $C$ is a constant independent of $\e$. 
The assumption $\min_{1\leq i\leq n} d_i=0$ guarantees that $C_{\mathcal F}>0$, since in this case either $I_2\neq \emptyset$ or $I_3\neq \emptyset$. 
Therefore, choosing $\lambda$ in \eqref{stimaG} so that $0<\lambda<C_{\mathcal F}/C_{\mathcal G}$, 
we obtain that the right-hand side of \eqref{esplode} tends to $+\infty$ as $\epsilon\to 0$. 
The proposition is proved.
\end{proof}

\subsection{Proofs of Theorem \ref{thGcn} and of Corollary \ref{confinon}}\label{nonhaunnome}
We combine the previous results to prove the main results in the case of $n$ dislocations.
\begin{proof}[Proof of Theorem \ref{thGcn}] 
By \eqref{nonmiscordare}, combining Propositions \ref{prop1} and \ref{prop2} with Propositions \ref{propR}, \ref{propG}, and \ref{ultima-prop} 
yields the continuous convergence of $\mathcal F_\e^{(n)}$ to $\mathcal F^{(n)}$.
The continuity of $\mathcal F$ follows from the relationship between continuous convergence and $\Gamma$-convergence (see \cite[Remark 4.9]{dalmaso}), which implies that $\mathcal F^{(n)}(\ba)$ tends to $+\infty$ as either one of the $\a_i$'s approaches the boundary or any two $\a_i$ and $\a_j$ (with $i\neq j$) become arbitrarily close.
Therefore $\mathcal F^{(n)}$ is minimized by $n$-tuples $\ba=(\a_1,\ldots,\a_n)$ of distinct points in $\Omega{}^n$.
\end{proof}

\begin{proof}[Proof of Corollary \ref{confinon}]
Recalling definition \eqref{dconi}, let $\delta>0$ and let 
\begin{equation*}
\Omega^{\delta,n}\coloneqq\{(x_1,\ldots,x_n) \in\Omega^n : d_i>\delta \ \ \forall i=1,\ldots,n\}.
\end{equation*}
The energy functional $\mathcal E_\e^{(n)}$ is continuous over $\overline{\Omega}{}^{\delta,n}$, for every $\e\in(0,\delta)$.
Indeed, given two different configurations of dislocations $\ba'$ and $\ba''$, it is possible to construct a diffeomorphism $\Phi$ that maps $\overline\Omega_\e({\ba'})$ into $\overline\Omega_\e({\ba''})$, keeps the boundary $\partial\Omega$ fixed, 
and satisfies $\|D\Phi - I\|_{L^\infty(\Omega;\R^{2{\times}2})}$, $\|\det D\Phi- 1\|_{L^\infty(\Omega)}= o(|\ba'-{\ba}''|)$.
This implies that $\mathcal E_\e^{(n)}({\ba'}) = \mathcal E_\e^{(n)}({\ba''}) + o(|\ba'-{\ba''}|)$. 

Fix now $\e>0$ and $k\in \mathbb N$ and let $\ba_{\epsilon,k}$ be such that
 \begin{equation}\label{8ott}
 \inf_{\overline\Om{}^n} \mathcal F_\epsilon^{(n)}\leq \mathcal F_\epsilon^{(n)}(\ba_{\epsilon,k})\leq \inf_{\overline\Om{}^n} \mathcal F_\epsilon^{(n)}+1/k.
 \end{equation}
 Without loss of generality, we can assume that the whole sequence $\ba_{\e,k}$ converges to some $\ba^\e\in \overline\Om{}^n$, as $k\to \infty$, and satisfies $|\ba_{\e,k}-\ba^\e|< 1 /k$.
 We claim that there exist $\overline{\e}>0$ and $\delta>0$ such that (defining $d_i^\e$'s as in \eqref{dconi}, associated with the point $\ba^\e$)
 \begin{equation}\label{ticonfino}
d^\e\coloneqq \min_i d_i^\e>\delta\qquad \forall\ \epsilon\in (0,\overline{\e}).
 \end{equation}
Assume by contradiction that there exists a subsequence of $\e$ (not relabeled) such that $d^\e\rightarrow0$ as $\eps\rightarrow0$. 
Let $k_\e$ be a sequence of natural numbers, increasing as $\e$ goes to zero, and let $\ba$ be a cluster point of the family $\{\ba_{\e,k_\e}\}_\e$. 
In view of Theorem \ref{thGcn}, thanks to \cite[Corollary 7.20]{dalmaso}, we infer that $\ba$ is a minimizer of the functional $\mathcal F^{(n)}$.
By the triangle inequality, we get
$$
\min_i d_i \leq d^\e +  | \ba^\e -\ba_{\e,k_\e}| + |\ba_{\e,k_\e} - \ba| < d^\e + \frac{1}{k_\e} + o(\e)\rightarrow0,
$$
which implies that there exists an index $i$ such that either $a_i\in \partial \Om$ or $a_i=a_j$ for some $i\neq j$, in contradiction with Theorem \ref{thGcn}.

Let $\ezero:=\min \{\overline\e,\delta\}$ and $\e\in (0,\ezero)$. 
In view of \eqref{ticonfino}, the minimizing sequence $\{\ba_{\e,k}\}_k$ lies in the set  $\{(x_1,\ldots,x_n) \in\Omega^n : d_i>\delta \ \ \forall i=1,\ldots,n\}$. 
Therefore, by \eqref{8ott}, we conclude that 
 \[
 \inf_{\overline\Om{}^n} \mathcal F_\epsilon^{(n)}=\lim_{k\to\infty} \mathcal F_\epsilon^{(n)}(\ba_{\epsilon,k})=\mathcal F_\epsilon^{(n)}(\ba^\epsilon),
 \]
namely $\ba^\e$ is a minimizer of $\mathcal F_\epsilon^{(n)}$ made of $n$ distinct points each of which is at distance at least $\delta$ from the boundary.
Eventually, again by \cite[Corollary 7.20]{dalmaso}, we infer that, up to a subsequence (not relabeled), $\ba^\e$ converges to a minimizer of $\mathcal F^{(n)}$ so that $\mathcal F_\e(\ba^\e)^{(n)}\to\mathcal F^{(n)}(a)$ as $\e\to0$.
\end{proof}

\section{Plots of the limiting energy}\label{numerica}

We plot here the limiting energy $\cF_\e^{(n)}$ of \eqref{Fepsn} in the special case $n=1$ and $\Omega$ the unit disk centered at the origin $O$ for two different boundary conditions.
In view of Corollary \ref{confinon}, for a fixed $\e>0$ small enough, the minima of $\mathcal F_\e^{(1)}$ are close to those of $\mathcal F^{(1)}$, therefore, by \eqref{Fepsn}, the energy landscape provided by $\mathcal F^{(1)}$ is a good approximation of that of $\mathcal E_\e^{(1)}$.

The two different boundary conditions that we consider are $f_1=1$ on $\partial\Omega$, and $f_2=2$ on $\partial\Omega\cap\{x>0\}$ and $f_2=0$ on $\partial\Omega\cap\{x<0\}$; the choices that we make for the numerics are 
$$g_1=\theta_O \qquad\text{and}\qquad
g_2=\begin{cases}
2\theta_O & \text{in the first quadrant,} \\
\pi & \text{in the second and third quadrants,} \\
2\theta_O-2\pi & \text{in the fourth quadrant.}
\end{cases}
$$

\begin{figure}[h]
\begin{center}
\includegraphics[scale=.18]{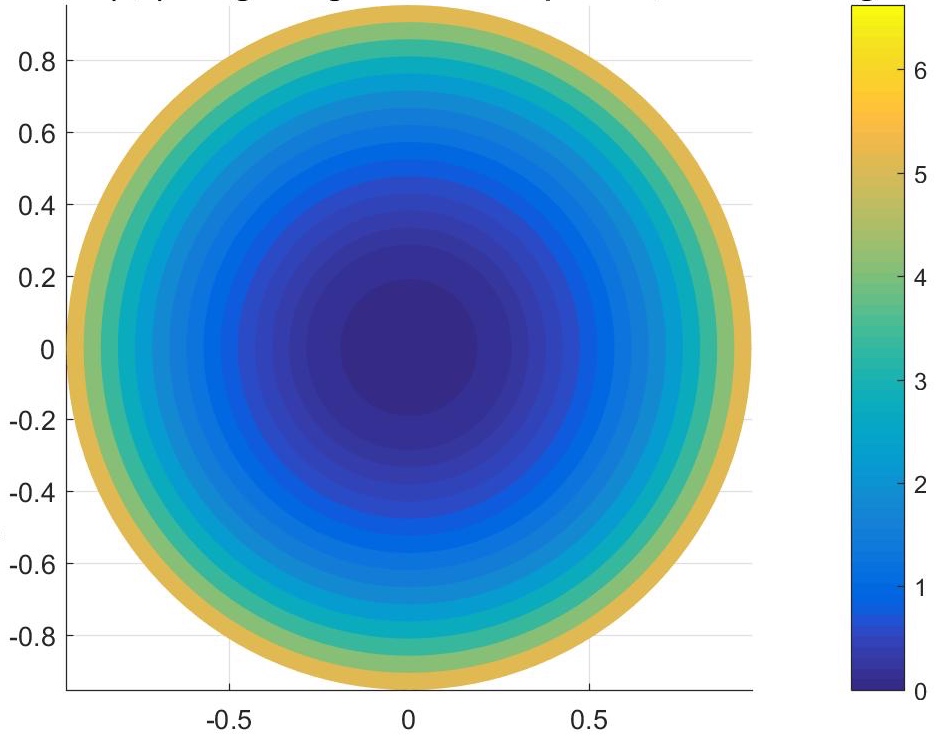}\qquad
\includegraphics[scale=.18]{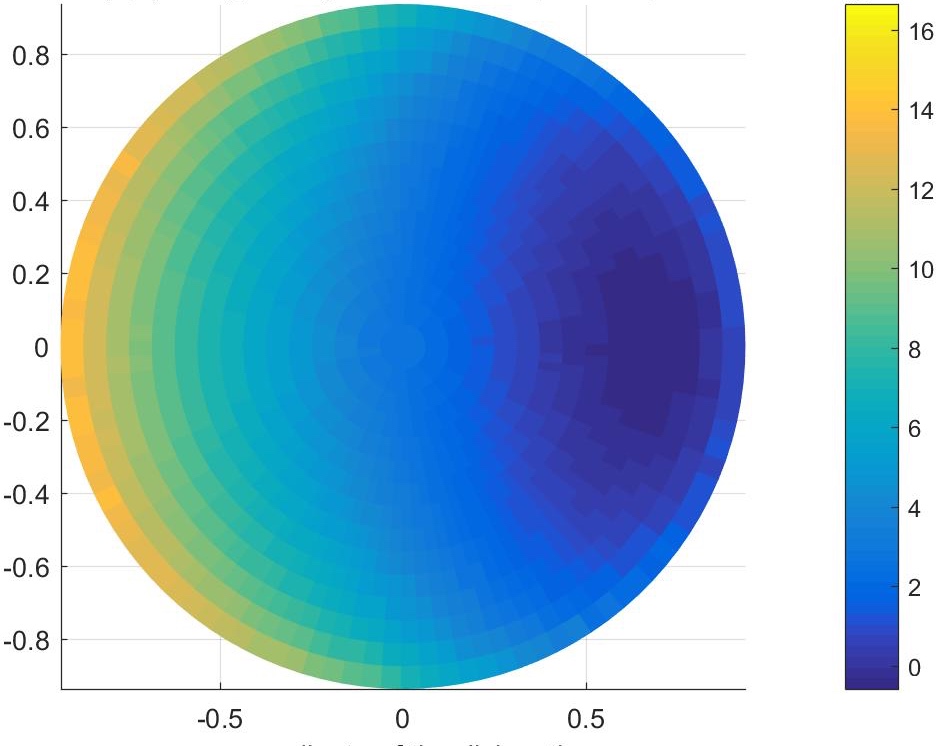}
\end{center}
\caption{To the left, the functional $\mathcal F^{(1)}$ associated with $g_1$; to the right, the functional $\mathcal F^{(1)}$ associated with $g_2$. Both are top views.}
\label{figone}
\end{figure}
From Figure \ref{figone} one can deduce that the energy profile associated with $f_1$ is radially symmetric and has a minimum in the center.
The one for $f_2$ is no longer radially symmetric, but symmetric with respect to the $y$ axis and has a minimum in the interior of $\Omega$ which is located where $f_2$ is the largest (by symmetry, it is located along the $x$-axis, at about $x=0.65$).

For $n\geq 2$, we consider $\Omega$ the unit disk centered at the origin, $f^{(n)}=n$ on $\partial \Omega$, and, for simplicity, we minimize $\mathcal F^{(n)}$ in a subclass of configurations: the vertices of regular $n$-gons centered at the origin (this particular choice is supported by the conjecture in \cite{SS2000}). The optimization problem becomes (numerically) easy, as it is enough to minimize the energy with respect to the circumradius of the $n$-gon, which may vary from 0 to 1.

\begin{figure}[h]
\begin{center}
\includegraphics[scale=.15]{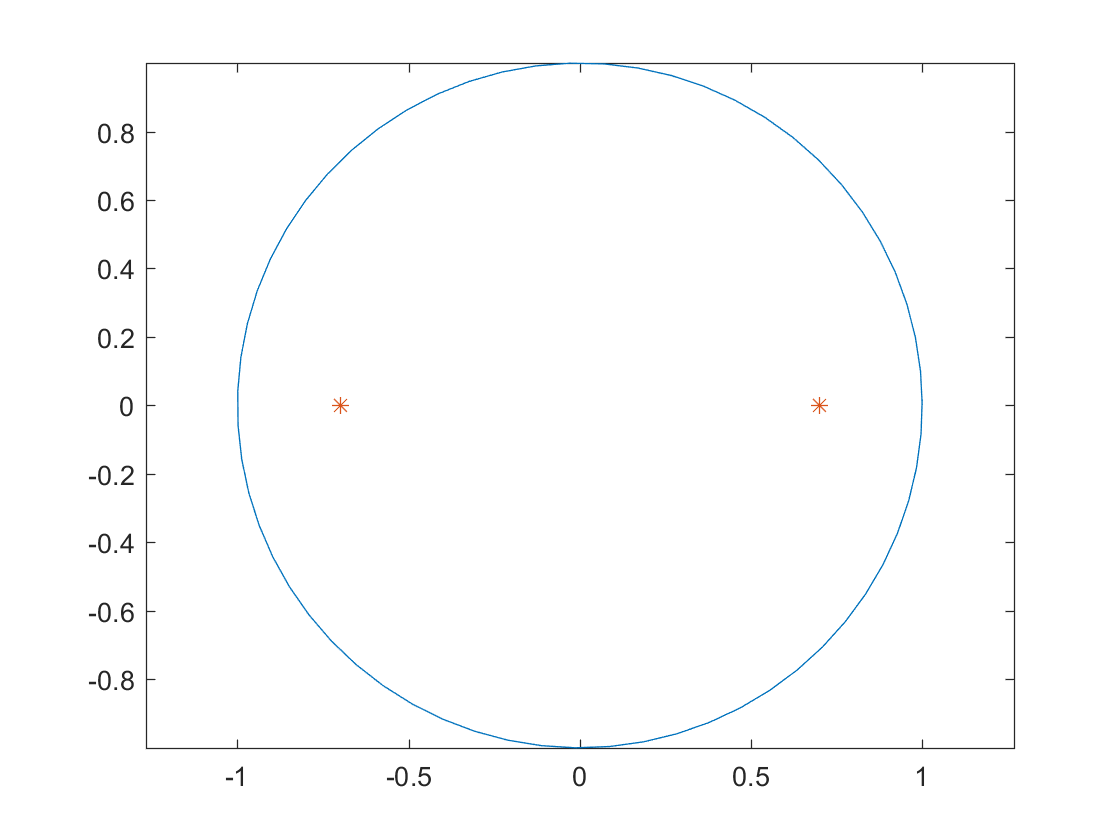}
\includegraphics[scale=.15]{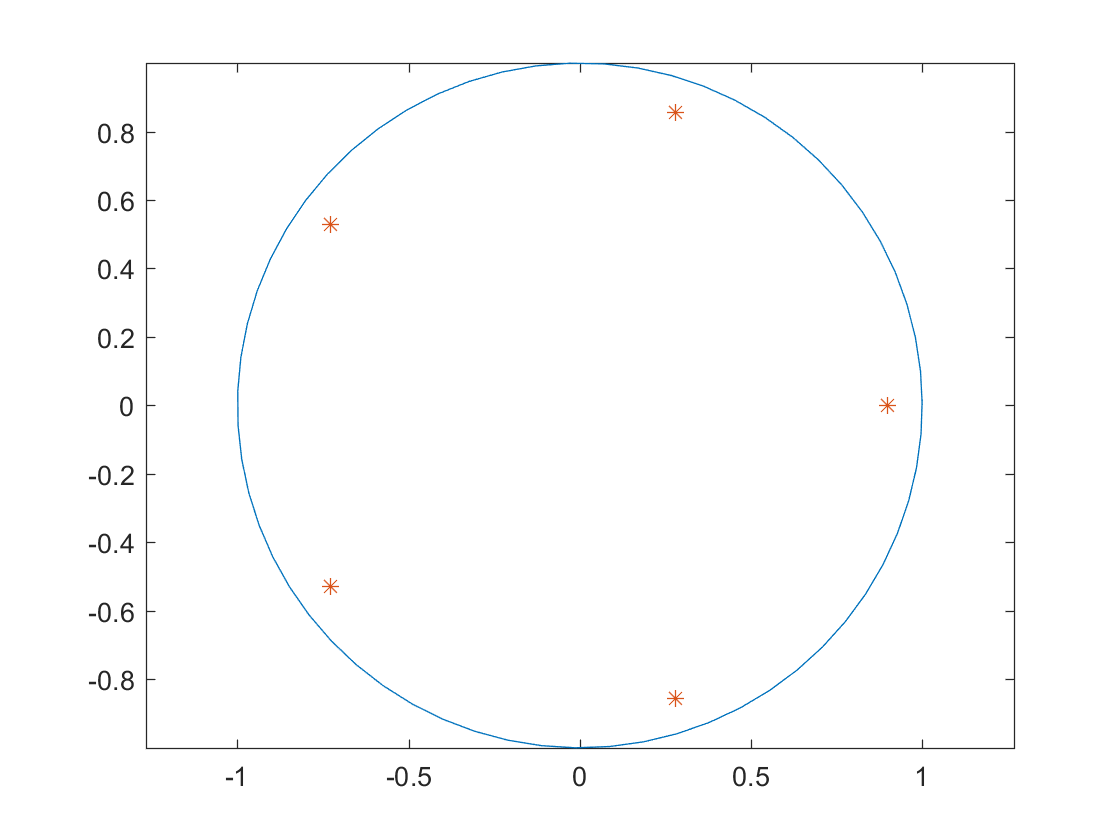}
\includegraphics[scale=.15]{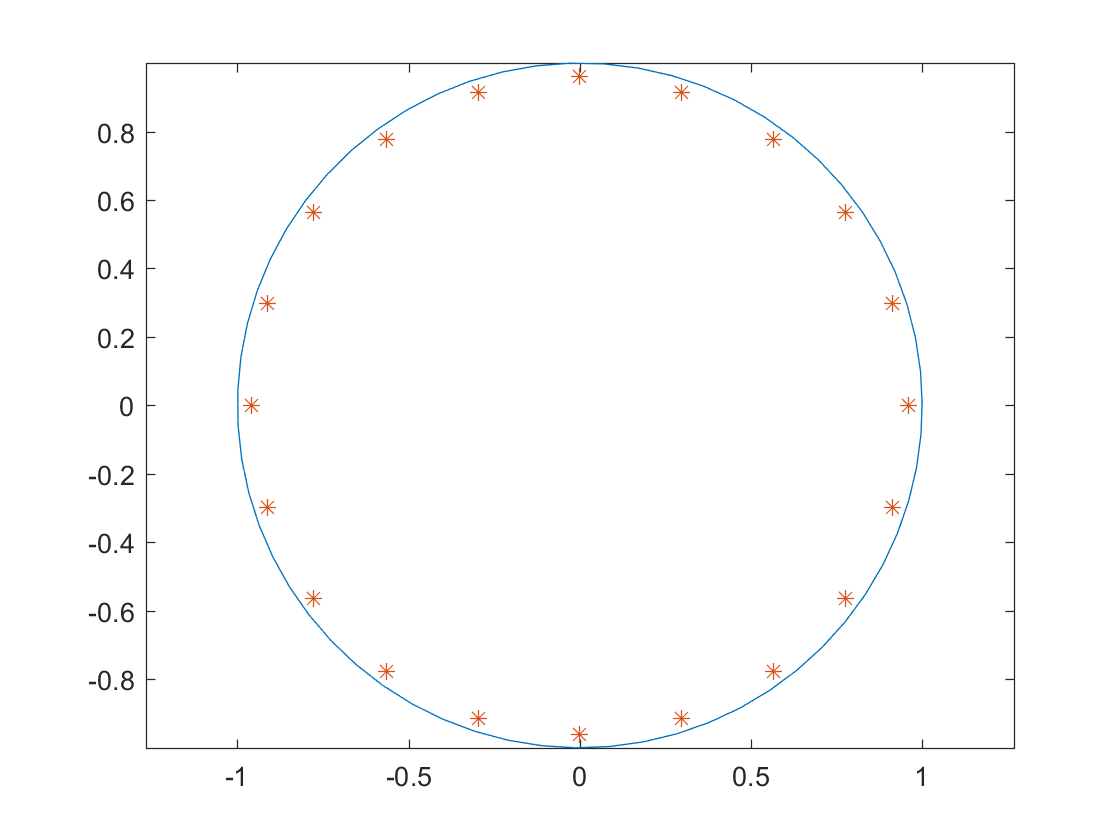}
\end{center}
\caption{The optimizers of $\mathcal F^{(n)}$ among the vertices of regular $n-$gons centered at the origin for $n=2, 5, 20$, from left to right.}
\label{plot2520}
\end{figure}

Let us denote by $\mathbf{p}_n$ an optimal $n$-gone of $\mathcal F^{(n)}$ (unique up to rotations). 
As it is suggested by Figure \ref{plot2520}, as $n$ increases we observe a decrease of the distance $\mathrm{dist}\, (\mathbf{p}_n, \partial \Omega)$ and of the energy $\mathcal F^{(n)}(\mathbf{p}_n)$ (which is negative). More precisely, we observe the following behaviors (see Figure~\ref{asy}):
$$
\mathrm{dist}\,(\mathbf{p}_n,\partial \Omega) \sim \frac{1}{n}\,,\quad |\mathcal F^{(n)}(\mathbf{p}_n)|\sim n^2\,,\quad \hbox{as }n\to +\infty\,.
$$
Let us mention that the investigation of minima and minimizers of $\mathcal F^{(n)}$ in the limit as $n\to +\infty$ is the object of our paper \cite{LMSZupscaling}.

\begin{figure}[h]
\begin{center}
\includegraphics[scale=.22]{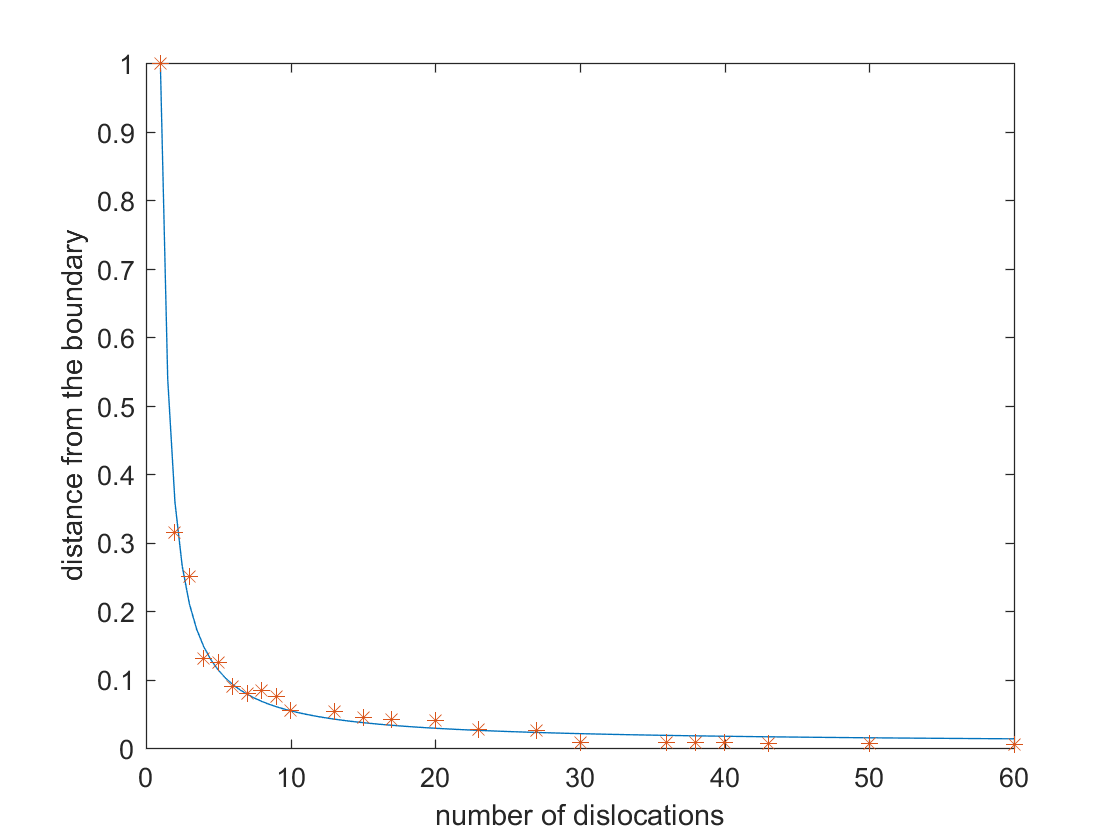}\qquad
\includegraphics[scale=.22]{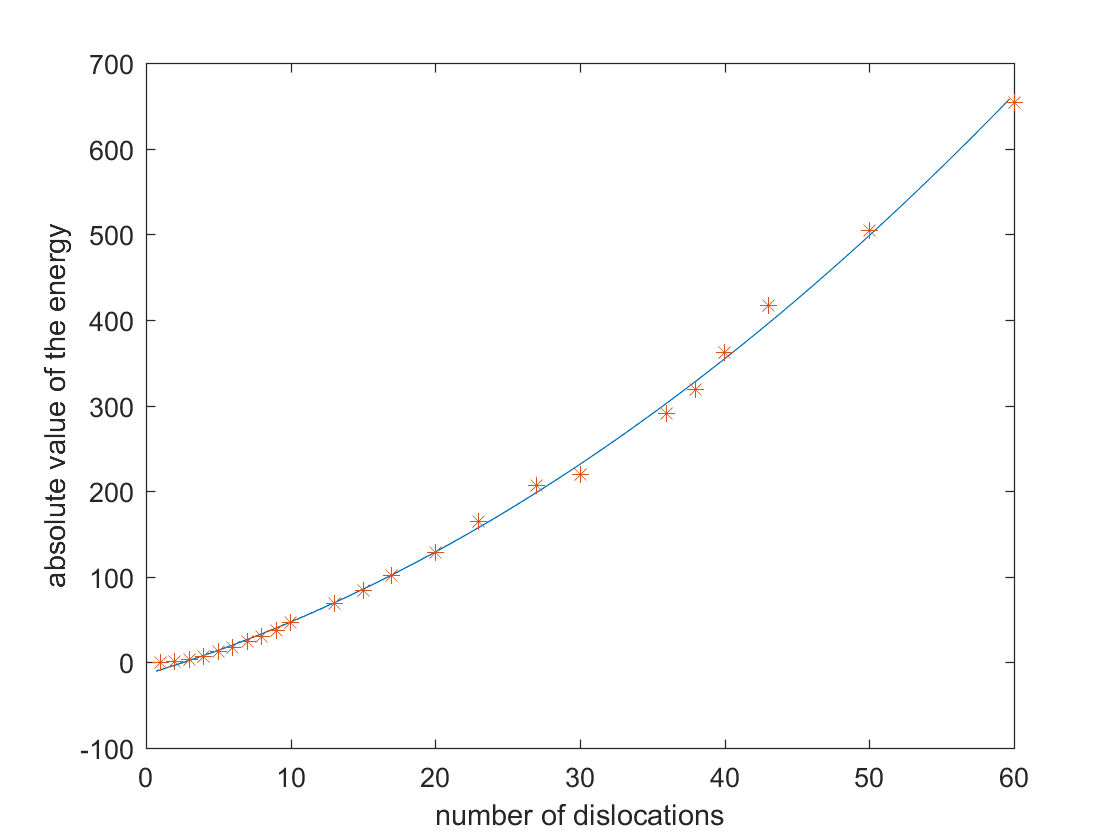}
\end{center}
\caption{Plot of $\mathrm{dist}\,(\mathbf{p}_n,\partial \Omega)$ (left) and plot of $|\mathcal F^{(n)}(\mathbf{p}_n)|$ (right) as functions of $n$. Test done for $n$ dislocations, $n$ between 1 and 60.}
\label{asy}
\end{figure}

\smallskip

\subsection*{Acknowledgements}
The authors are members of the INdAM-GNAMPA project 2015 \href{http://fcm2.weebly.com/}{\emph{Fenomeni Critici nella Meccanica dei Materiali: un Approccio Variazionale}}, which partially supported this research.
I.L.\@ and M.M.\@ acknowledge partial support from the ERC Advanced grant \href{https://people.sissa.it/~dalmaso/QuaDynEvoPro-Home.htm}{\emph{Quasistatic and Dynamic Evolution Problems in Plasticity and Fracture}} (Grant agreement no.: 290888).
M.M.\@ acknowledges partial support from the ERC Starting grant \emph{High-Dimensional Sparse Optimal Control} (Grant agreement no.: 306274) and the DFG Project \emph{Identifikation von Energien durch Beobachtung der zeitlichen Entwicklung von Systemen} (FO 767/7).
R.S.\@ acknowledges partial support from the ERC Starting grant FP7-IDEAS-ERC-StG (EntroPhase) (Grant agreement no.: 256872). R.S. is also grateful to the Erwin  Schr\"odinger Institute for the financial support obtained during the last part of the present research.
D.Z.\@ acknowledges partial support from the FIR Starting grant \href{https://sites.google.com/site/firb2013geoqualaspectspde/home}{\emph{Geometrical and qualitative aspects of PDE's}}.

\end{document}